\documentclass[11pt]{amsart}
\usepackage{amssymb,amsthm,amsmath,amstext}
\usepackage{graphicx}
\usepackage{mathrsfs}
\usepackage{bm}
\usepackage{multirow}
\usepackage[all]{xy}
\usepackage{tikz-cd}
\usepackage{xfrac}
\usepackage{anyfontsize}
\usepackage{longtable}
\usepackage{ulem}

\usepackage[left=1.25in,top=1.35in,right=1.25in,bottom=1.35in]{geometry}

\usepackage{color}

\theoremstyle{plain}
\newtheorem{thm}{Theorem}\newtheorem*{thm*}{Theorem}
\newtheorem*{conj*}{Conjecture}

\newtheorem{pr}[thm]{Proposition}
\newtheorem{lem}[thm]{Lemma}
\newtheorem{cor}[thm]{Corollary}

\theoremstyle{remark}

\theoremstyle{definition}

\newtheorem{exam}[thm]{Example}

\newtheorem{rem}[thm]{Remark}

\newcommand{\CC}{\mathbb{C}}

\newcommand{\PP}{\mathbb{P}}
\newcommand{\GG}{\mathbb{G}}

\newcommand{\reg}{\operatorname{reg}}
\newcommand{\tth}{\operatorname{th}}
\newcommand{\Sing}{\operatorname{Sing}}
\newcommand{\codim}{\operatorname{codim}}
\newcommand{\mult}{\operatorname{mult}}
\newcommand{\Ver}{\operatorname{Vert}}
\newcommand{\Bs}{\operatorname{Bs}}
\newcommand{\red}{\operatorname{red}}

\newcommand{\inv}{{-1}}
\newcommand{\rk}{\operatorname{rk}}
\newcommand{\HO}{\operatorname{H}^0}
\newcommand{\Hom}{\operatorname{Hom}}

\newcommand{\SSx}{\operatorname{II}_{x,X}}
\newcommand{\Sc}{\operatorname{Sc}}
\newcommand{\map}{\dasharrow}
\newcommand{\Pic}{\operatorname{Pic}}
\newcommand{\vvert}{\operatorname{Vert}}
\newcommand{\length}{\operatorname{length}}
\newcommand{\Tan}{\operatorname{Tan}}
\newcommand{\Sec}{\operatorname{Sec}}
\newcommand{\Bl}{\operatorname{Bl}}
\newcommand{\tht}{\operatorname{th}}
\usepackage[backref=page]{hyperref}
\hypersetup{pdftitle={}}
\hypersetup{pdfauthor={}}
\hypersetup{colorlinks=true,linkcolor=blue,anchorcolor=blue,citecolor=blue}

\begin{document}

\title[On the tangent degree and the degree of the tangent  variety]{On the tangent degree and the degree of the tangent variety of a projective variety}
\author[J. Hernandez Gomez]{Jordi Hernandez Gomez}
\author[F. Russo]{Francesco Russo*}
\address{Dipartimento di Matematica e Informatica, Universit\` a degli Studi di Catania, 
Viale A. Doria 5, 95125 Catania, Italy
}
\email{francesco.russo@unict.it,  jordi.emanuel.hernandez.gomez@gmail.com}
\thanks{*Both authors were partially  supported  by the PRIN 2022 {\it Birational geometry of moduli spaces and special varieties}. The second author is a member of the G.N.S.A.G.A. of INDAM}

 

\maketitle
 \begin{center}{\it Dedicated to Ciro Ciliberto on the 
occasion of his $75^{\tht}$ birthday}
\end{center}

\begin{abstract}
The tangent degree $\tau(X)$ of a projective variety $X^n\subset\PP^N$ is the number of tangent spaces  to $X$ at smooth points 
passing through a general point of the tangent variety $\Tan(X)\subseteq\PP^N$, if positive and finite; it is equal to zero if $\dim(\Tan(X))<2n$. 

In this paper we focus on  general properties of $\tau(X)$ and of $\deg(\Tan(X))$. 
For example $\tau(X)\neq 1$ if $N=2n$ and, as soon as $\Tan(X)$ does not coincide with the secant variety, we prove a linear lower bound for the degree of $\Tan(X)$ in terms of its codimension in the spirit of \cite{Ciliberto.Russo.2006}. Then we consider the cases in which the previous two invariants attain the lower bounds found here, either in small dimension/codimension and/or under the smoothness assumption. Finally for $N\geq 2n+1$ we consider varieties
$X^n\subset\PP^N$ having $\tau(X)>1$ and provide their classification in small dimension.

\end{abstract}
\section*{Introduction}

The starting point of the topics we consider  in this paper is the higher dimensional generalization of the following elementary facts: if $X\subset\PP^2$ is an irreducible curve, then 
$\tau(X)\geq 2$ if $\deg(X)\geq 2$ and $\tau(X)=2$ if and only if $X$ is a conic. 
We first prove that  $\tau(X)\neq 1$ for an irreducible variety $X^n\subset\PP^{2n}$ and that if $\tau(X)=2$, then $X$ is  birational to its general tangent space via an  explicit map (see Theorem \ref{If n>=3 then tau(X)>=2} and Lemma \ref{tau(X)=2 implies X rational}) and hence it is rational. Let us remark that if  the singular locus of $X^n\subset\PP^{2n}$ consists of at most a finite number of {\it nodes}, then $\tau(X)\neq 1$ can be deduced  in various ways. For example, under these hypothesis Severi's Double Point Formula in \eqref{Sev:formula} implies that $\tau(X)$ is even  (see also Remark \ref{Johnsontau}), which is clearly false for arbitrary singular varieties. This is analogous  to the case of plane curves, where Pl\" ucker's  Formulas imply that $\tau(X)$ is even for curves having  only nodes as singularities. Then we concentrate on the classification of  varieties $X^n\subset\PP^{2n}$ with $\tau(X)=2$ and $n\geq 2$. After showing the existence of (singular) rational scrolls in $\PP^4$ of arbitrary large degree having $\tau=2$ (see Example \ref{Verra}), we construct other (singular) examples of arbitrary large dimension and degree, pointing out interesting connections with notable classes of algebraic varieties like $OADP$-varieties, $QEL$-varieties of type $\delta=1$ or twisted cubics over Jordan algebras (see Examples \ref{QELex} and \ref{X333}).

We prove   that smooth varieties $X^n\subset\PP^{2n}$ with $\tau(X)=2$ are $QEL$-varieties of type $\delta=1$ and viceversa (see Proposition \ref{smoothtau2}) and conclude that these varieties are either
smooth rational normal scrolls $S(1,\ldots, 1,2)\subset\PP^{2n} $ or odd--dimensional Fano manifolds $X^{2m+1}\subset\PP^{4m+2}$, $m\geq 1$, with $\Pic(X)\simeq\mathbb Z\langle\mathcal O_X(1)\rangle$
and of index $m+1$ (the latter exist at least for  $m=1,2$). In particular, the unique even--dimensional smooth varieties $X^{2m}\subset\PP^{4m}$ with $\tau(X)=2$ are rational normal scrolls $S(1,\ldots, 1,2)$ (see Corollary \ref{classtau2smooth}).
We then prove that irreducible surfaces $X^2\subset\PP^4$ with $\tau(X)=2$ are rational scrolls, which can have arbitrary large degree  if they are singular (see Theorem \ref{class:omega_2=2}).

 In the last sections we focus on the degree of the tangent variety $\Tan(X)$ of a  variety $X^n\subset\PP^N$, $N\geq 2n+1$, and specifically  on the cases in which $\tau(X)\geq 2$. If $\Tan(X)$ coincides with the secant variety $\Sec(X)$, then  \cite[Theorem 4.2]{Ciliberto.Russo.2006} (see Theorem \ref{Lower bound on deg(SX))}) provides a lower bound for $\deg(\Tan(X))$ which is quadratic in the codimension of $\Tan(X)$ in $\PP^N$. If $\Tan(X)\subsetneq\Sec(X)$, we show that   $\deg(\Tan(X))\geq 2(N-\dim(\Tan(X))+1)$ (see Theorem \ref{Lower bound on deg(TX) when TX is not SX)}) and that this  lower bound is sharp (see Example \ref{ex:minimalTan} and  also Remark \ref{lowTan}).

We proceed  with the construction of non-developable $X^n\subset \PP^N$, $N\geq 2n+1$, contained in a developable scroll in planes $Y^{n+1}\subset\PP^N$, showing that there exist $X^n\subset Y^{n+1}\subset\PP^N$ with $\tau(X)$ arbitrary large. In most cases such varieties are singular with the notable exceptions of the so called {\it Roth varieties}, for which $Y=S(L,Z)$ is a cone with vertex a line $L$ over a smooth rational normal scroll $Z^{n-2}\subset\PP^{N-2}$ (see Theorem \ref{Roth:Ilic}).
Surfaces $X^2\subset \PP^N$, $N\geq 5$, with $\tau(X)\geq 2$ are classified in Theorem \ref{class:surf:taugeq2}, proving that either $X$ is a Veronese surface in $\PP^5$ or $X^2\subset Y^3\subset\PP^N$ 
with $Y^3$ a developable scroll in planes over a curve such that the general plane  intersects $X$ in a curve of degree at least two. In the smooth case the previous result is refined in Corollary \ref{smooth:taugeq2}, showing that there exist only two classes of examples: the Veronese surface in $\PP^2$ ($\tau=2$) and Roth surfaces  $X^2\subset S(L,Z)\subset \PP^N$ of degree $b(N-2)+1$ for some $b\geq 2$  and having $\tau=b^2$. 

The main result of  \cite{Ilic.1998} is a characterization of Roth varieties  as the unique smooth linearly normal varieties
$X^n\subset \PP^N$, $N\geq n+2$, such that the double divisor of a general projection onto $\PP^{n+1}$ is not ample. Mumford proved  that these divisors vary in a base-point free linear system (see \cite[Introduction]{Ilic.1998}). Corollary \ref{smooth:taugeq2} provides a different geometric characterization of Roth surfaces, leading to the natural question if this is also true for  smooth linearly normal non secant--defective varieties $X^n\subset\PP^N$ with  $N\geq 2n+1$ and $n\geq 3$.
As far as we know this problem has not been previously addressed in the literature (neither for $n=2$).

In the first section we present a general introduction to the theory of foci of the family of tangent spaces to an algebraic variety and we collect various facts regarding the second fundamental form of an algebraic variety and developable varieties. Since  these results are usually dispersed in the literature we thought that it would be  useful to present them in an unitary and concise way, 
without any pretension of completeness nor of originality and trying to quote explicitly all the references we are aware of. This  vast area of research, lying at the crossroads of algebraic and differential geometry, has  many long-standing contributions so  we apologize in advance if we forgot some of them.

{\bf Acknowledgements.} The second author is grateful to Anand Patel, who a long time ago asked to him about the existence of smooth surfaces in $\PP^N$ with $\tau\geq 2$, $N\geq 5$  and different from the Veronese surface. This question led to Theorem \ref{class:surf:taugeq2} and  Corollary \ref{smooth:taugeq2}, which have been the starting motivations of this paper. He also wishes to thank Ciro Ciliberto and Luca Chiantini for some discussions about the above question, which appear now as some of the preliminary results in the last section.

\section{Preliminaries}
Let $X\subset \PP^N$ be a quasi projective variety of dimension $n\geq1$, that we assume to be irreducible, unless otherwise stated, and defined over $\mathbb C$. For any $x\in X$ we denote by $T_x X\subseteq \PP^N$ the projective tangent space of $X$ at $x$ and by $t_xX$ the Zariski tangent space to $X$ at $x$. The locus of smooth points of $X$ is the open non-empty subset $X_{\reg}\subseteq X$.

\subsection{Foci of families of linear spaces in $\PP^N$}

Let $V$ be a smooth quasi-projective variety of dimension $n\geq 1$ and let 
\begin{equation} \begin{tikzcd}\label{fam:lin:n:r}
	{\mathcal V} & { \PP^N} \\
	{V}
	\arrow["q", from=1-1, to=1-2]
	\arrow["p"', from=1-1, to=2-1]
\end{tikzcd}
\end{equation}
be an  {\it $n$-dimensional family of linear spaces of dimension $r\geq 1$ in $\PP^N$}. This means that  $p:\mathcal V\to V$
is a smooth projective morphism
such that $p^\inv(v)\simeq\PP^r$ for every $v\in V$ and that $q(p^\inv(v))=\PP^r_v\subset\PP^N$ is a linear subspace
of $\PP^N$ of dimension $r$ for every $v\in V$. The morphism $q$ is called the {\it tautological morphism} of the family $p:\mathcal V\to V$.
It is important to remark that we are not assuming that the morphism $f:V\to \mathbb G(r,N)$ induced by the universal property
of the Hilbert scheme is injective.

A point $s\in \mathcal V$ (or $q(s)\in\PP^N$) is called a {\it focal point} or a {\it focus}, for short, of  $p:\mathcal V\to V$ if the differential of $q$
at $s$ is not injective. If $p(s)=v$, we shall sometimes use the notation $s\in\PP^r_v$.

A point $s\in q(\mathcal V)\subset\PP^N$ is called a {\it fundamental point} of $p:\mathcal V\to V$ if $\dim(q^\inv(s))>0$ (in classical language
one says that through $s$ there pass infinitely many elements of the family). Clearly, a fundamental point is a focal
point but the converse is not true in general.

If $\dim(q(\mathcal V))<\dim(\mathcal V)=n+r$, then every $s\in \mathcal V$ is a focus. In this case the notion is not interesting
and such families will be called {\it degenerated}.
From now on we shall suppose that the family $p:\mathcal V\to V$ is non-degenerate, i.e. that $\dim(q(\mathcal V))=\dim(\mathcal V)=n+r$ (hence necessarily  $N\geq n+r$). 

By definition the quasi-projective variety $\mathcal V$ is smooth and irreducible. By applying generic smoothness to $q:\mathcal V\to\PP^N$, the non-degenerate assumption assures that the ramification locus of $q$ does not contain a general $\PP^r_v=p^\inv(v)$. The set of focal points in $\mathcal V$, which will be also called the focal locus, is a closed Zariski subset and, under our hypothesis, for $v\in V$ general it intersects  $\PP^r_v$ in a proper closed subset.

The characteristic map of $p:\mathcal V\to V$ at $v\in V$  is  the differential of the induced map $f:V\to \mathbb G(r,N)$ at $v$:
$$(df)_v: t_vV\to \HO(N_{\PP^r_v/\PP^N})\simeq t_{[\PP^r_v]}\mathbb G(r,N),$$
where $N_{\PP^r_v/\PP^N}$ is the normal bundle of $\PP^r_v$ in $\PP^N$.
The globalization of this map to $\mathcal V\subset V\times \PP^N$
 induces  a morphism of sheaves on $\PP^r_v$ by restriction to the fiber $\PP^r_v$:
\begin{equation}\label{charmor}
\chi(v):t_vV\otimes \mathcal O_{\PP^r_v}\to N_{\PP^r_v/\PP^N}\simeq \mathcal O_{\PP^r_v}(1)^{N-r},
\end{equation}
which is injective for $v\in V$ general by the non-degenerate assumption (see \cite{Ciliberto.Sernesi.2010} for details). Since $N\geq n+r$, then $N-r\geq n$ and it is not difficult to see that $s\in \PP^r_v$ is a focus of $p:\mathcal V\to V$ if and only if  the induced morphism $\chi(v)_s: k(s)^n\to k(s)^{N-r}$ satisfies $\rk(\chi(v)_s)< n$ (see also \cite{Ciliberto.Sernesi.2010} and  \cite[\S 2.2.4]{Fischer.Piontkowski.Book2001}, revisiting in modern terms C. Segre's approach in \cite{Segre1}). In particular for $N-r=n$, or equivalently for $\overline{q(\mathcal V)}=\PP^N$, the locus of focal points in $\PP_v^r$ is the support of a hypersurface of degree $n$.

 The map  $p:\mathcal V\to V$ is, up to shrinking $V$,  locally trivial  so that we can suppose that for  any  $v\in V$ and  $s\in \PP^r_v$ we have 
 \begin{equation}\label{t-splitting}
 t_s\mathcal V\simeq t_vV\times t_s\PP^r_v.
 \end{equation} 
 Since $q$ restricted to $\PP^r_v$ is an isomorphism,   $s\in\PP^r_v$ is a focus if and only if there exists a non zero vector $(\mathbf w,\mathbf 0)\in t_s\mathcal V$ such that  $(dq)_s(\mathbf w,\mathbf 0)=\mathbf 0$. 
 \medskip
  
In the sequel we shall  explain in which sense, according to the classical terminology used by C. Segre, a focus $s\in\PP^r_v$ as above lies  in an infinitely near $r$-plane  of the family $p$  in the direction $\mathbf w\in t_vV$. We shall start with the easiest example, well-known in elementary differential geometry but  useful to understand some 
non-standard definitions introduced later.

Let $\mathcal L(A)$ denote the linear subspace generated by a set $A$ in a vector space and let us denote by $\langle A\rangle\subseteq\PP^N$  the linear span of a
subset $A\subset\PP^N$. If $A=\{a_1,a_2\}$, then  $\langle a_1,a_2\rangle:=\langle\{a_1,a_2\}\rangle$.

\begin{exam}\label{dev:surf} Let $p:\mathcal S\to C$ be a non-degenerate one-dimensional family of lines in $\PP^N$, $N\geq 3$,  over a smooth curve $C$ and let $q:\mathcal S\to S\subset\PP^N$
be the associate tautological morphism, which we suppose to be  birational onto its  image $S$ (this implies that  the associated map $f:C\to \widetilde C\subset \mathbb G(1,N)$ is also birational onto its image $\widetilde C$).
Let $c_0\in C$ be a general point and let $l_0=\PP^1_{c_0}\subset\PP^N$ be the corresponding line. We shall study  $S$ locally around a (smooth) point $s_0\in l_0\subset S\subset \PP^N$.

To this aim we can suppose that $S$ has, locally around $s_0\in l_0$, a parametrization $\mathbf x:\Delta\to \mathbb C^{N+1}$, $\Delta\subset\mathbb C^2$,  of the the form
$$\mathbf x(u,v)=a(u)+v\cdot  b(u),$$
i.e. we are supposing that $u$ is a coordinate in a neighbourhood of $c_0$ centred at $0\in C$ and that $l_u=\PP(\mathcal L(a(u),b(u)))\subset\PP^N$ are the lines near $l_0$ in $S$ (the value $v=\infty$ corresponds to the point at infinity of $l_u$).

The projective tangent plane to $S\subset\PP^N$ at a  general point $\PP(\mathbf x(0, v))$ is  
$$T_{(0,v )}S=\PP(\mathcal L(a(0)+ v\cdot b(0), b(0), a^\prime(0)+v\cdot b^\prime(0))),$$ which clearly coincides with $\PP(\mathcal L(a(0), b(0), a^\prime(0)+v\cdot b^\prime(0)))$. 

The ruled surface $S$ is said to be {\it developable along $l_0$} if the tangent space at the smooth points of $S$ is constant along the smooth points on $l_0$, i.e. if it does not depend on $v$. This is clearly equivalent to the condition $\dim(\mathcal L(a(0),b(0),a^\prime(0),b^\prime(0)))=3$. The case $\dim(\mathcal L(a(0),b(0),a^\prime(0),b^\prime(0)))=2$ is excluded because otherwise $q$ would not be non-degenerate, contradicting our hypothesis.

If  $\dim(\mathcal L(a(0),b(0),a^\prime(0),b^\prime(0)))=4$, it is not difficult to see that in this case for all $v\in\mathbb C\cup\infty$ the differential $(dq)_{(0,v)}:t_{(0,v)}\mathcal S\to t_{(0,v)}S$ is an isomorphism because $$a^\prime(0)+v\cdot b^\prime(0)\not\in\mathcal L(a(0),b(0))$$ for every $v\in \mathbb C\cup\infty$.

If  $\dim(\mathcal L(a(0),b(0),a^\prime(0),b^\prime(0)))=3,$ then there exists a unique $\overline v\in \mathbb C\cup \infty$ such that $\dim(\mathcal L(a(0),b(0),a^\prime(0)+\overline v\cdot b^\prime(0)))=2$. In this case there exists a unique focus $s=(0,\overline v)\in l_0$.
If the surface  $S$ is developable along a general line $l_c=\PP^1_c\subset S$, $c\in C$, then there exists a unique focus $s_c\in l_c$. If the focus $s_c$ does not move, then  the lines $l_c$  pass through a fixed point $v_0\in S$ and $S$ is a cone. If the focus $s_c$ moves, then the points $s_c$ describe the so-called {\it edge of regression}
of the developable ruled surface $S\subset\PP^N$. In both cases   $s_c$ is a singular point of $S$ for any $c\in C$. 

If the surface $S\subset\PP^N$ is developable along a general line $l_c$, then the surface is called developable and the corresponding  one-dimensional family of lines $p:\mathcal S\to C$ will be called  {\it developable}.

The previous condition of developability has nice geometrical interpretations, one of them  due to Max N\" other. 

Let hypothesis be as above. 
Up to shrinking $C$, we can suppose that $\widetilde{C}=f(C)\subset \mathbb G(1,N)$
is smooth and  that $f(c_0)=[l_0]$ is a non-singular point of $\widetilde C$.  Let us recall that 

$$t_{[l_0]}\mathbb G(1,N)\simeq \Hom(\mathcal L(a(0),b(0)), \frac{\mathbb C^{N+1}}{\mathcal L(a(0),b(0))}).$$

Then it is not difficult to realize that  $t_{[l_0]}\widetilde C$ corresponds to the non-zero  homomorphism
$$\alpha\cdot a(0)+\beta\cdot b(0)\mapsto [\alpha\cdot a^\prime(0)+\beta\cdot b^\prime(0)].$$

Hence $S$ is developable along $l_0$ if and only if $t_{[l_0]}\widetilde C$ is not injective. In this case the focus on $l_0$ corresponds to the unique $\overline v\in\mathbb C\cup \infty$  determined by  $\PP(\ker(t_{[l_0]}\widetilde C)).$ We shall generalize  this example in \S \ref{dev:fam}.
\end{exam}

\subsection{Family of tangent spaces to a variety in $\PP^N$ and tangent variety}\label{subss:famtang}

Let  $${T^\circ_X}:=\{(x,s)\in  X_{\reg}\times \PP^N : s\in T_x X\}\subset X_{\reg}\times \PP^N$$ 
and let $$T_X:=\overline{T_X^\circ}\subseteq X\times \PP^N.$$ We have the following diagrams:
\begin{equation}\label{tanfam}
\begin{tikzcd}
	{T^\circ_X} & { \PP^N} \\
	{X}_{\reg}
	\arrow["q", from=1-1, to=1-2]
	\arrow["p"', from=1-1, to=2-1]
\end{tikzcd}
\end{equation}
where $p$ and $q$  are the restrictions of the projections onto the first and second factor, respectively. Analogously we have  $\overline p:T_X\to X$ and
$\overline q:T_X\to\PP^N$.  The {\it tangent variety of $X$} is  
$$\Tan(X):=\overline q(T_X) =\overline{\bigcup_{x\in X_{\reg}}T_x X}\subseteq \PP^N,$$
while  $p:T_X^\circ\to X_{\reg}$ is called the {\it family of tangent  spaces to $X\subset\PP^N$}.  Clearly $\Tan(X)$ is irreducible and $\dim(\Tan(X))\leq\min\{2n,N\}$.
The {\it secant variety to $X\subset\PP^N$} is defined as
$$\Sec(X):=\overline{\bigcup_{x_1\neq x_2\, ,\, x_i\in X}<x_1,x_2>}\subseteq \PP^N.$$
Then $\Sec(X)\subseteq\PP^N$ is  irreducible  and $\dim(\Sec(X))\leq\min\{2n+1,N\}$. Moreover, since a  line tangent at $x\in X_{\reg}$ is a limit of secant lines $\langle x_1, x_2\rangle$ with $x_i\to x$, we have 
the inclusion $\Tan(X)\subseteq \Sec(X)$.

The  family $p:T_X\to X$ induces the morphism $f:X_{\reg}\to \mathbb G(n,N)$  given by the rule $f(x)=[T_xX]$ and thus it  defines  the so-called {\it Gauss map of $X$}: 
 $$\mathcal{G}_X : X\dashrightarrow \GG(n,N).$$ 
 Let   $\widetilde{X}=\overline{{\mathcal G}_X(X_{\reg})}\subset \mathbb G(n,N)$ be the {\it Gauss image of $X$}.

 Letting  $\widetilde{T}:=\PP({\mathcal{S}}_{|\widetilde{X}})$ with $\mathcal{S}$ the universal locally free sheaf of rank $n+1$ over $\GG(n,N)$, we have
 the  diagram 
\begin{equation}\label{Tuniv}
\begin{tikzcd}
	{\widetilde{T}} & {\PP^N} \\
	{\widetilde{X}}
	\arrow["\widetilde{q}", from=1-1, to=1-2]
	\arrow["\widetilde{p}"', from=1-1, to=2-1]
\end{tikzcd},
\end{equation}
  where $\widetilde{p}: \widetilde{T}\to\widetilde{X}$ is the projective bundle morphism and $\widetilde{q}:\widetilde{T}\to\PP^N$ is the tautological morphism . Since $\Tan(X)=\widetilde{q}(\widetilde{T})\subseteq \PP^N$, we get the finer estimate $\dim(\Tan(X))\leq\dim({\widetilde X})+n$.

A variety $X\subset\PP^N$ of dimension $n\geq 1$  is said to be {\it developable} if $\dim(\widetilde{X})<\dim(X)$.

We  collect a series of well-known results/remarks for further reference.

\begin{rem}
Let  $\varrho=\varrho(X):=\dim(X)-\dim(\widetilde{X})\geq 0$ be the {\it Gauss defect of $X$}. It is well-known that the closure of a general fiber of $\mathcal{G}_X$ is a linear space $\PP^\varrho\subset X\subset\PP^N$ (see for example \cite[Theorem~1.5.10]{Russo.Book2016}). In particular, either $\mathcal{G}_X:X\dashrightarrow \widetilde{X}$ is birational or $X$ is developable and a general tangent space is constant along an open subset of a linear space of dimension $\varrho$. For a non-developable variety $X\subset\PP^N$, the birationality onto the image of $\mathcal G_X$ assures that  the two families $\overline p:T_X\to X$  and $\widetilde{p}:\widetilde{T}\to \widetilde X$ coincide in a neighbourhood  of a general point $x\in X$. 

If $U\subseteq X_{\reg}$ is isomorphic to its image $\widetilde U\subseteq\widetilde X$
via $\mathcal G_X$, then the family of tangent spaces is locally trivial over $U$, being isomorphic over $U$ to  $\PP(\mathcal G_X^*({\mathcal{S}}_{|\widetilde{U}}))\to U$. For smooth varieties $X\subset\PP^N$ the locally free sheaf $\mathcal G_X^*({\mathcal{S}}_{|\widetilde X})$ of rank $n+1$ is isomorphic to  the  {\it sheaf of principal parts of $(X,\mathcal O(1))$}, which can be defined directly on $X$ as the extension of $\Omega_X^1\otimes O(1)$ by $\mathcal O(1)$ determined by $c_1(\mathcal O(1))\in H^1(\Omega_X^1)$:
\begin{equation}\label{PXext}
0\to\Omega_X^1(1)\to \mathcal P_X\to \mathcal O(1)\to 0,
\end{equation}
yielding $\det(\mathcal P_X)\simeq\omega_X(n+1)$ and that \eqref{PXext} is not splitting.
\end{rem}

\begin{rem}

The  above terminology  generalizes the one  introduced in Example \ref{dev:surf} for the case of surfaces. For a developable surface $S\subset\PP^N$ as in Example \ref{dev:surf} we have that $\widetilde S$ is the closure of $\widetilde C$ in $\mathbb G(1,N)$, that $\varrho(S)=1$ and that the closure of a general fiber of $\mathcal G_S$ is $l_c$.

An irreducible  curve $X\subset\PP^N$ is developable if and only if it is a line. Developable varieties of dimension two and three have been classified. A developable surface $S\subset\PP^N$, $N\geq 3$, is either a plane, a cone over a non-degenerate curve or the tangent variety to a non-degenerate curve (see for example \cite[Proposition~2.5.6]{Fischer.Piontkowski.Book2001}). We shall come back to the three dimensional case later.

If $X\subset\PP^N$ is developable, then necessarily $\dim(\Tan(X))\leq\dim(\widetilde X)+n<2n$. The converse is not true since there exist smooth varieties $X\subset\PP^N$
of dimension $n\geq 3$ such that $N\geq 2n$ and  $\dim(\Tan(X))<2n$. For example the tangent variety to the Segre fourfold $X=\PP^2\times\PP^2\subset\PP^8$ is
a cubic hypersurface.
\end{rem}

\subsection{Focal locus of the  family of tangent spaces to a variety}\label{parTanX}
The study of the focal and  the fundamental locus of the family of projective tangent spaces to a variety $X\subset\PP^N$ of dimension $n\geq 2$ deserves a
special treatment for various peculiarities. 

From now on we shall suppose $N\geq 2n$ and that the family of tangent spaces to $X\subset\PP^N$ is non-degenerate, i.e.  $\dim(\Tan(X))=2n$. In particular,  $X\subset\PP^N$ is non-developable. Under these hypothesis, the morphisms $\overline{q}:T_X\to \Tan(X)$ and $q:T^\circ_X\to \PP^N$ are generically \' etale. In analogy with the situation illustrated in Example \ref{dev:surf} we shall analyse  in detail the focal locus of $p:T_X^\circ\to X_{\reg}$.

We  study $X\subset\PP^N$ locally around a general fixed $x\in X_{\reg}$.
Let $\PP:\mathbb C^{N+1}\setminus\mathbf 0\to \PP^N$ be the quotient morphism. We can suppose $x=\PP(\mathbf e_0)$.  
Let 
\begin{equation}\label{parx}
\mathbf \Phi:\Delta^n\to U\subset \mathbb C^N=\PP^N\setminus V(x_0)
\end{equation}
be a local holomorphic parametrization  of $X$   defined on a sufficiently small polydisc  $\Delta^n\subset \mathbb C^{n}$  with coordinates $\mathbf u=(u_1,\ldots ,u_n)$ centered  at $\mathbf 0\in\mathbb C^n$ and such that $x=\mathbf \Phi(\mathbf 0)$.  If $\mathbf \Phi=(\Phi_1,\ldots,\Phi_N)$, then the Jacobian of the  functions $\Phi_i$ is  of maximal rank $n$ at every point $\mathbf u\in\Delta^n.$

Let 
$$\partial_i\mathbf \Phi(\mathbf u)=(\frac{\partial \mathbf \Phi_1}{\partial u_i}(\mathbf u), \ldots, \frac{\partial \mathbf \Phi_N}{\partial u_i}(\mathbf u)).$$ Then 
$t_{\mathbf \Phi(\mathbf u)}X\subset\mathbb C^N$ 
has a parametrization $\phi_{\mathbf u}:\mathbb C^n\to t_{\mathbf \Phi(\mathbf u)}X$ given
by $$\phi_{\mathbf u}(\mathbf{\lambda})=\mathbf \Phi(\mathbf u)+\sum_{i=1}^n\lambda_i\cdot \partial_i\mathbf \Phi(\mathbf u),$$
where $\mathbf{\lambda}=(\lambda_1,\ldots, \lambda_n)\in\mathbb C^n$.
Let us recall the identifications  $$t_{\phi_{\mathbf u}(\lambda)} T_{\Phi(\mathbf u)} X=t_{\phi_{\mathbf u}(\lambda)} t_{\Phi(\mathbf u)} X=t_{\Phi(\mathbf u)}X$$ 
so that  $$t_{(\Phi(\mathbf u),\phi_{\mathbf u}(\lambda))} T^\circ_X =t_{\Phi(\mathbf u)} X \times t_{\phi_{\mathbf u}(\lambda)}T_xX=t_{\Phi(\mathbf u)} X \times t_{\Phi(\mathbf u)}X.$$
For every $j=1,\ldots, n$, let
$$\mathbf w_j(\lambda,\mathbf u)=(\sum_{i=1}^n\lambda_i\cdot \frac{\partial^2\mathbf \Phi_1}{\partial u_i\partial u_j}(\mathbf u), \ldots, \sum_{i=1}^n\lambda_i\cdot \frac{\partial^2\mathbf \Phi_N}{\partial u_i\partial u_j}(\mathbf u)).$$

The image of 
$$dq_{(\Phi(\mathbf u),\phi_{\mathbf u}(\lambda))}: t_{\Phi(\mathbf u)} X \times t_{\Phi(\mathbf u)}X\rightarrow t_{\phi_{\mathbf u}(\lambda)} \Tan(X)\subseteq 
\mathbb C^N$$
is the affine subspace passing through $\mathbf\phi_{\mathbf u}(\lambda)$ and having direction equal to 
\begin{equation}\label{diffG}
\mathcal L(\partial_1\mathbf\Phi(\mathbf u),\ldots, \partial_n\mathbf\Phi(\mathbf u), \mathbf w_1(\lambda,\mathbf u), \ldots, \mathbf w_n(\lambda,\mathbf u)),
\end{equation}
whose dimension is  
$$n+\dim(\mathcal L( \mathbf w_1(\lambda,\mathbf u), \ldots, \mathbf w_n(\lambda,\mathbf u)))\leq 2n.$$

The (second) {\it osculating space $T^{(2)}_xX$ at $x$} is the projective closure of the affine space through $x=\Phi(\mathbf 0)$ with direction as in \eqref{diffG} with  $(\lambda,\mathbf u)=(\mathbf 0, \mathbf 0).$ Using higher order derivations one defines $T^{(k)}_xX$ for every $x\in X$ and for every $k\geq 1$ in an analogous way. With this definition clearly $T^{(1)}_xX=T_xX$ for any $x\in X$.  

We have proved the following result, already known (at least) to Terracini (see \cite{Terracini}).

\begin{pr}\label{focuscar} Let notation be as above and suppose $\lambda\neq\mathbf 0$. Then the following conditions are equivalent:
\begin{enumerate}
\item $(x,\phi_{\mathbf 0}(\lambda))\in x\times t_xX\subset x\times T_xX$ is a focus of the family of projective tangent spaces;
\medskip
\item $$\dim(\mathcal L( \mathbf w_1(\lambda,\mathbf 0), \ldots, \mathbf w_n(\lambda,\mathbf 0)))<n;$$
\medskip
\item There exists $$\mathbf v=\sum_{j=1}^n\mu_j\cdot \partial_j{\mathbf \Phi}(\mathbf 0)\in t_xX\setminus x$$
such that
$$\sum_{j=1}^n\mu_j\cdot \mathbf w_j(\lambda,\mathbf 0)=\mathbf 0;$$
or, equivalently, such that
\begin{equation}\label{sym2}
\sum_{i,j=1}^n\lambda_i\cdot \frac{\partial^2\mathbf \Phi_k}{\partial u_i\partial u_j}(\mathbf 0)\cdot\mu_j=0
\end{equation}
 for every $k=1,\ldots, N.$
\medskip

\end{enumerate}

Moreover, $(x,\phi_{\mathbf 0}(\lambda))\in x\times t_xX$, $\lambda\neq\mathbf 0$,  is a focus of the family of  tangent spaces to $X$ if and only if $(x,\phi_{\mathbf 0}(\alpha\cdot \lambda))\in x\times t_xX$ is a focus of the family of  tangent spaces to $X$ for every $\alpha\in\mathbb C^*$.

In particular, if $(x,\phi_{\mathbf 0}(\lambda))\in x\times t_xX$ is a focus, then the line $\langle x, \phi_{\mathbf 0}(\lambda)\rangle\subset T_xX$ consists of foci and the support of the intersection of the focal locus of the family of tangent spaces to $X$ with $T_xX$ is a cone with vertex $x$.
\end{pr}

\begin{rem}\label{singTX)}
Generic smoothness  and  \eqref{diffG}  yield  the following dimension formula  for an arbitrary $X\subset\PP^N$:
\begin{equation}\label{dimTXsecond}
\dim(\Tan(X))=n+\dim(\mathcal L( \mathbf w_1(\lambda,\mathbf u), \ldots, \mathbf w_n(\lambda,\mathbf u)))
\end{equation}
for $(\lambda,\mathbf u)\in \mathbb C^n\times \Delta$ general and also
\begin{equation}\label{dimT2Xsecond}
\dim(T_x^{(2)}X)=n+\dim(\mathcal L( \mathbf w_1(\mathbf 0,\mathbf 0), \ldots, \mathbf w_n(\mathbf 0,\mathbf 0)))
\end{equation}
for $x\in X_{\reg}.$
Moreover, the previous computations also show  that $\Tan(X)$ is singular along $X$ if $\Tan(X)\subsetneq \PP^N$ (see also  \cite[Proposition~3.3.3]{Fischer.Piontkowski.Book2001}).  In  general, the singular locus of $\Tan(X)$ can be bigger than $X$. The next well-known result is again a direct consequence of the above analysis (see also  \cite[Proposition~3.3.1]{Fischer.Piontkowski.Book2001}).
\end{rem}

\begin{pr}\label{TX is developable} Let $X\subset \PP^N$ be  a non-degenerate variety.
 Then $\Tan(X)\subseteq\PP^N$ is  a developable variety. More precisely, if $y\in \Tan(X)$ is general, i.e. $y\in T_x X$ is general with $x\in X$ general, then the projecitve tangent space to $\Tan(X)$  remains constant along $\langle x,y\rangle \cap \Tan(X)_{\reg}$ and  $T_xX\subset T_z\Tan(X)$ for $z\in\langle x,y\rangle$ general.
 \end{pr}

\subsection{Second fundamental form of a projective variety}
Let notation and hypothesis be as above and let $\widetilde{\mathbf \Phi}:\mathbb C^*\times \Delta^n\to \PP^\inv(U)=\mathbb C^*\times U\subset\mathbb C^{N+1}\setminus\mathbf 0$ be the lifting of $\mathbf \Phi$ to the quotient map  $\PP:\mathbb C^{N+1}\setminus\mathbf 0\to \PP^N$ given by
$$\widetilde{\mathbf \Phi}(\alpha,\mathbf u)=(\alpha,\alpha\cdot \mathbf \Phi(\mathbf u)).$$

The {\it normal space to $X\subset\PP^N$ at $x$} is 
$$N_xX=\frac{t_x\PP^N}{t_xX}\simeq\frac{t_x\mathbb C^N}{t_xX},$$
 where $\mathbb C^N=\PP^N\setminus V(x_0).$
 
 If   $T_{\mathbf \Phi(\mathbf u)}X=\PP(W_{\mathbf \Phi(\mathbf u)})$ with $W_{\mathbf \Phi(\mathbf u)}\subset\mathbb C^{N+1}$, then $\widetilde{\mathbf \Phi}(\alpha,\mathbf u)\in W_{\mathbf\Phi(\mathbf u)}$ for every $(\alpha,\mathbf u)\in \mathbb C^*\times\Delta^n$ and  
$$t_xX\simeq\frac{W_{\mathbf \Phi(\mathbf u)}}{\langle \widetilde{\mathbf\Phi}(1,\mathbf u)\rangle}\mbox{, respectively } t_x\PP^N\simeq \frac{\mathbb C^{N+1}}{\langle 
\widetilde{\mathbf\Phi}(1,\mathbf u)\rangle}.$$ 
Hence  
$$t_{[T_{\mathbf \Phi(\mathbf u)}X]}\mathbb G(n,N)\simeq \Hom(W_{\mathbf \Phi(\mathbf u)},\frac{\mathbb C^{N+1}}{W_{\mathbf \Phi(\mathbf u)}})\simeq \Hom(t_xX,N_xX).$$ 

 Let $\pi:t_x\mathbb C^N\to N_xX$ be the quotient map and let 
 $\varphi:t_xX\to \Hom(t_xX,N_xX)$ be the linear map defined by 
 $$\varphi(\lambda)=\pi\circ dq_{(x,\phi_{\mathbf 0}(\lambda))_{t_xX\times \mathbf 0}}.$$
 If $$\mathbf v=\sum_{j=1}^n\mu_j\cdot \partial_j{\mathbf \Phi}(\mathbf 0)\in t_xX$$
 and if $\mu=(\mu_1,\ldots, \mu_n),$
then 
\begin{equation}\label{defdGX}
\varphi(\lambda)(\mathbf v)=\varphi(\lambda,\mu):=[\sum_{j=1}^n\mu_j\cdot \mathbf w_j(\lambda,\mathbf 0)]\in N_xX.
\end{equation}

Since $\phi_{\mathbf 0}$ is linear, the linear map $\varphi$ induces a bilinear map 
$$\widetilde{\varphi}_x:t_xX\times t_xX\to N_xX,$$
 which is symmetric by \eqref{sym2}. The image of the dual  linear map $(N_xX)^*\to S^2(t_xX^*)$ is a vector space of homogenous quadratic forms on $t_xX$.

The associated linear system of quadrics $|\SSx|$ on $\PP(t_xX)\simeq \PP^{n-1}$ (or the above symmetric bilinear map $\widetilde{\varphi}_x$) is called the {\it second
fundamental form of $X$ at $x$}. By definition $\dim(|\SSx|)\leq N-n-1$ and, letting $\varphi_x:\PP(t_xX)\dasharrow |\SSx|^*$ be the associated rational map, formula \eqref{dimTXsecond}
translates into
\begin{equation}\label{dimTXssecond}
\dim(\Tan(X))=n+1+\dim(\overline{\varphi_x(\PP(t_xX)})
\end{equation}  
(see also {\cite[$\S$~5.7, p. 408]{Griffiths.Harris.1979} or \cite[Theorem~3.3.1, p. 129]{Fischer.Piontkowski.Book2001}}). 
In particular for $N\leq 2n$, $\dim(\Tan(X))=N$ implies $\dim(|\SSx|)= N-n-1$.

\begin{rem}\label{remdGX} 
The map \eqref{defdGX} is nothing but the differential of  the Gauss map, i.e.   the linear map
$$(d\mathcal G_X)_x:t_xX\to t_{[T_xX]}\mathbb G(n,N)\simeq \Hom(t_xX,N_xX).$$

As above, suppose  that $N\geq 2n$, yielding $\dim(N_xX)=N-n\geq n$, and also that $\dim(\Tan(X))=2n$. Then Proposition \ref{focuscar} can be reinterpreted in the following way.  Let $\mathbf w\in t_xX\setminus x$ and let $(d\mathcal G_X)_x(\mathbf w): t_xX\to N_xX$ be the corresponding linear map. The map $(d\mathcal G_X)_x(\mathbf w)$ is not injective if and only if $(d\mathcal G_X)_x(\alpha\cdot \mathbf w)$ is not injective for every $\alpha\in\mathbb C^*$ if and only if the line $\langle x,\mathbf w\rangle$ consists of foci of the family of tangent linear spaces to $X\subset\PP^N$ (see also \cite[$\S$~5.5, p. 408]{Griffiths.Harris.1979}).

Let $\mathbf w\in t_xX\setminus x$, let  $Q\in |\SSx|$ be a quadric hypersurface and let   $H_{[\mathbf w],Q}\subset\PP(t_xX)$ be the polar hyperplane of $Q$ with respect to $[\mathbf w]$. Then
$$\PP(\ker((d\mathcal G_X)_x(\mathbf w)))=\bigcap_{Q\in |\SSx|}H_{[\mathbf w],Q}.$$
In particular $\mathbf w\in t_xX\setminus x$ is a focus if and only if the intersection of the polar hyperplanes of the quadrics in $|\SSx|$ with respect to $[\mathbf w]$ is not empty. Clearly this intersection may be computed as  the intersection of  the polar hyperplanes of  a basis of the second fundamental form.
\end{rem}

If  $T_xX=\PP(\mathcal L(e_0)\oplus t_xX)$, then  $\PP(t_xX)$ can be identified with the {\it hyperplane at infinity} for the previous splitting and $t_xX\subset T_xX$ with the affine space $T_xX\setminus \PP(t_xX)$.
Hence the points in $\PP(t_xX)$ correspond to the directions through $x$ of the lines in $t_xX$ passing through $x$. For this reason  we shall say that, for $\mathbf v\in t_xX\setminus x$, the  point $[\mathbf v]\in\PP(t_xX)$ is the direction  of the line $\langle x,\mathbf v\rangle \subset t_xX$ (or of the line $\langle x,[\mathbf v]\rangle \subset T_xX$). 

Two directions $[\mathbf v], [\mathbf w]\in \PP(t_xX)$ are said to be {\it conjugate} if $(d\mathcal G_X)_x(\mathbf v)(\mathbf w)=\mathbf 0$, which by symmetry  occurs if and only if 
$(d\mathcal G_X)_x(\mathbf w)(\mathbf v)=\mathbf 0.$  This means that they are {\it conjugate} (or {\it polar}) for every quadric $Q\in |\SSx|$.

A direction $[\mathbf v]\in\PP(t_xX)$ is said to be an {\it asymptotic direction} if $ (dG_X)_x(\mathbf v)(\mathbf v)=\mathbf 0$.
Using the parameterization \eqref{parx} and the definition of the second fundamental form we deduce that a direction $[\mathbf v]\in \PP(t_xX)$ is asymptotic if and only if the multiplicity of intersection  of the 
tangent line $\langle x, [\mathbf v]\rangle$ with $X$ at $x$  is at least 3. Equivalently, $[\mathbf v]\in\PP(t_xX)$ is an asymptotic direction if and only if $[\mathbf v]\in\Bs(|\SSx|)$, the base locus  of the second fundamental form of $X$ at $x$.

\begin{pr}\label{Key observation}
    Let $x\in X\subset\PP^N$ be a general point and suppose that $r\in T_x X\setminus x$ is a fundamental point for the family of tangent spaces to $X$.
    Then $\langle x, r\rangle\cap \PP(t_xX)\in \Bs(|\SSx|)$.
\end{pr}
\begin{proof} Suppose $r\in T_x X\setminus x$ is a fundamental point  and let $\langle x, r\rangle\cap \PP(t_xX)=[\mathbf w]$. 
Let $\widetilde{C_r}\subset X_{\reg}\times r\subset T_X^\circ$ be a smooth curve such that $(x,r)\in \widetilde{C_r}$ and  let $\mathbf v\in t_xX$ be such that 
$$\mathcal L(\mathbf v)=t_{(x,r)}\widetilde{C_r}\subseteq \ker((dq)_{(x,r)})\simeq \ker((d\mathcal G_X)_x(\mathbf w)),$$
yielding $(d\mathcal G_X)_x(\mathbf w)(\mathbf v)=\mathbf 0$.
 Define $C_r:=p(\widetilde{C_r})\subset X$. Then  $x\in C_r$ and $\mathcal L(\mathbf v)=t_x C_r\subset t_x X$. If $\mathbf v\in \mathcal L(\mathbf w)$, then
 $\mathbf 0=(d\mathcal G_X)_x(\mathbf w,\mathbf w)$ and $[\mathbf w]\in\Bs(|\SSx|)$.
 
Suppose $\mathbf v\not\in \mathcal L(\mathbf w)$.  For each $y\in C_r\setminus x$ the line $L_y:=\langle y,r\rangle$ is tangent to $X$ at $y$, i.e. $L_y\subset T_y X$ and so $\mult_y(L_y \cap X)\geq 2$. When $y$ tends to $x$ along $C_r$, or equivalently, in the direction specified by $\mathbf v$, the line $L_y$ tends to the line $L_x=\langle x,r\rangle\subset T_x X$. Since $L_x$ is already tangent to $X$ at $x$ in a direction different from $\mathbf v$, we deduce $\mult_x(L_x \cap X)\geq 3$, yielding $[\mathbf w]\in \Bs(|\SSx|)$.
\end{proof}

We now  recall for further reference  an important  fact on the base locus  of the second fundamental form $|\SSx|$. We shall adopt the classical definition of {\it scroll over a curve of dimension $n$}, that is $X^n\subset\PP^n$ is the image of a 1-dimensional family of $\PP^{n-1}$'s such that through the general point of $X$ there passes  a unique $\PP^{n-1}$ of the family. In other words there exists a family $p:\mathcal X\to C$ of $\PP^{n-1}$'s over a smooth curve $C$ such that the tautological morphism $q:\mathcal X\to X$ is birational. 

\begin{thm}[{\cite[$\S$~3.21, p. 399]{Griffiths.Harris.1979}}]\label{Fixed component in Bs(|II|) implies X is a scroll over a curve}
  Let $X^n\subset \PP^{N}$ with $n\geq 3$. If for a general $x\in X$ the second fundamental form  $|\SSx|$ has a fixed component, then $X\subset\PP^N$ is a scroll in $\PP^{n-1}$'s over a curve and $\PP(t_x \PP^{n-1})=\PP^{n-2}$ is a fixed component of $|\SSx|$.

    Conversely, if $X^n\subset\PP^N$ is a scroll in $\PP^{n-1}$'s over a curve, then $\PP(t_x \PP^{n-1}_x)=\PP^{n-2}\subseteq \Bs(|\SSx|)$.
\end{thm}

\section{The $i^{\tth}$--ceto of an irreducible projective variety $X\subset\PP^N$}
Following Severi (see \cite{Severi}) we define the $i^{\tth}$-ceto of a projective variety $X^n\subset\PP^N$
of dimension $n\geq 1$ in the following way.

The $n^{\tth}$-ceto $\omega_n(X)$ of $X^n\subset\PP^N$ is the number of tangent spaces to $X$ at smooth points
meeting a general $2n$-codimensional linear subspace of $\PP^N$. The $i^{\tth}$-ceto $\omega_i(X)$
of $X^n\subset\PP^N$ is the $i^{\tth}$-ceto of the intersection of $X$ with a general $(n-i)$-codimensional linear subspace of $\PP^N$.

In particular  $\omega_n(X)=0$ if and only if $\dim(\Tan(X))<2n$.  Letting notation be as in \eqref{Tuniv}, we put $\tau(X):=\deg(\widetilde q)$, which we call the {\it tangent degree} of $X\subset\PP^N$. For $N\geq 2n$,  we have $$\omega_n(X)=\tau(X)\cdot\deg(\Tan(X)).$$ 
If $N\geq 2n$ and if we project $X\subset\PP^N$ from a general subspace $L^{N-2n}\subset \PP^N$, 
then $\omega_n(X)$ is equal to the number of {\it pinch points} of the projection $\pi_L(X)\subset\PP^{2n-1}$.

For a smooth projective variety  $X\subset\PP^N$, one can compute $\omega_n(X)$ via Segre classes
of the locally free sheaf  of principal parts $\mathcal P_X=\mathcal G_X^*(\mathcal S_{\widetilde X})$:
\begin{equation}\label{SegreclasseFormula}
\omega_n(X)=\mathcal O_{\PP(\mathcal P_X)}(1)^{2n}.
\end{equation}
The value of  $\tau(X)=\deg(\widetilde q)$ is harder to compute if $N>2n$.

Severi introduced these invariants to prove the {\it Double Point Formula} below. If $N\geq 2n$ and if $X^n\subset\PP^N$ is a smooth  variety of degree $d$ and if $\sigma(X)$ is the number of nodes of a general linear projection of $X$ into $\PP^{2n}$, then
\begin{equation}\label{Sev:formula}
2\cdot \sigma(X)=d\cdot (d-1)-\sum_{i=1}^n\omega_i(X)
\end{equation}
(see \cite{Severi} and \cite{Catanese} for a modern treatment). Let us remark that for $n=1$ \eqref{Sev:formula} is nothing but the usual Pl\" ucker Formula for curves with at most nodes as singularities 
since $\omega_1(X)=\deg(X^*)$ in this case.

If $\mu(X)\geq 1$ is the (finite) number of secant lines to $X$ passing through  a general point of $\Sec(X)\subset\PP^N$, called the {\it secant degree} of $X\subset\PP^N$, then $\sigma(X)=\mu(X)\cdot \deg(\Sec(X))$. One expects $\mu(X)=1$ as soon as $\Sec(X)\subsetneq\PP^N$ and varieties with $\mu(X)>1$ are special and nowadays called {\it weakly defective} according to the terminology introduced by Chiantini and Ciliberto in \cite{Chiantini.Ciliberto.wd}. One also expects $\tau(X)=1$ if $\tau(X)>0$ and $\Tan(X)\subsetneq\PP^N$. In \S~\ref{sec:Roth} we shall see examples of smooth varieties $X^n\subset\PP^N$ with $N>2n $ and $\tau(X)$ arbitrary large.

\subsection{Generalities on $\omega_n$ for the case $N=2n$}\label{N=2n}
For an irreducible curve $X\subset\PP^2$ we have $\tau(X)=0$
if and only if $X$ is a line and  $\tau(X)\geq 2$ for any curve $X\subset\PP^2$ of degree $d\geq 2$. Moreover,  $\tau(X)=2$ if and only if $d=2$, i.e.
if and only if $X\subset\PP^2$ is a smooth conic. 

We now study the case $n\geq 2$ to see to what extent one can generalize the previous facts.

\begin{thm}\label{If n>=3 then tau(X)>=2}
 Let $X^n\subset \PP^{2n}$ be a variety with $\Tan(X)=\PP^{2n}$. Then $\tau(X)\geq 2$.
\end{thm}
\begin{proof}
 Since $\Tan(X)=\PP^{2n}$ we have $\tau(X)\geq 1$. For $n=1$, the claim is true. Suppose now $n\geq 2$ and  $\tau(X)=1$. Then every focal point of the family of tangent spaces to $X$ is fundamental by Zariski's Main Theorem applied to $q:T_X\to \Tan(X)=\PP^{2n}$. 
     Let $x\in X$ be general and let $F_x \subset T_x X$ be the focal hypersurface. We know that $(F_x)_{\red}$ is a cone with vertex $x$ over a base $(\overline{F_x})_{\red}\subset\PP(t_xX)\subset T_xX$ with $\dim((\overline{F_x})_{\red})=n-2$ by Proposition \ref{focuscar}. 
     
     Moreover, since every point $r\in (F_x)_{\red}\setminus x$ is fundamental, we deduce $(\overline{F}_x)_{\red}\subseteq \Bs|\SSx|$ by Proposition \ref{Key observation}.
  If $n\geq 3$, then  $X$ is a scroll over a curve by Theorem \ref{Fixed component in Bs(|II|) implies X is a scroll over a curve} and  $(F_x)_{\red}=\PP_x^{n-1}\subseteq (T_x X\cap X)_{\red}\subset X$ is a general linear space in the ruling of $X$. Since through a general point of $\PP^{n-1}_x$ there are no other linear spaces of the ruling, we get a contradiction.

Suppose now $n=2$.    Since $\dim(\Tan(X))=4$, we have $\dim(|\SSx|)=1$ by  \eqref{dimTXssecond}, showing that $(F_x)_{\red}$ consists of a line $L_x$ tangent to $X$ at $x$.
 
    Let $F\subset\PP^4$ be the focal locus of $p:T^\circ_X\rightarrow X_{\reg}$. Then clearly $$X\subseteq \overline{\bigcup_{x\in X_{\reg}} L_x} \subseteq F.$$ Since $\dim(F)\leq 2$ by Zariski's  Main Theorem applied to $q:T_X\to \Tan(X)=\PP^4$, we have equality in the first inclusion. From this one deduces  that $X\subset\PP^4$ is indeed  a scroll over a curve because a line through $x$,  contained in $X$ and different from $L_x$ would produce another base point in $\Bs(|\SSx|)$, which is impossible. Then, as above, we get a contradiction.
\end{proof}

If $\Tan(X)=\PP^{2n}$, then $\Sec(X)=\PP^{2n}$. By Terracini's Lemma two general tangent spaces $T_{x_1}X$ and $T_{x_2}X$ intersect  at a single point. Fixing a general $x\in X$ we
 define the {\it Scorza map at $x$} for any variety $X^n\subset\PP^{2n}$ with $\Sec(X)=\PP^{2n}$:
$$\Sc_x:X\map T_xX,$$
by $\Sc_x(y)=T_yX\cap T_xX$  for $y\in Y$ general (see also \cite[\S 5]{Chiantini.Ciliberto.Russo} where this map is used in a related context). Clearly $\deg(\Sc_x)=\tau(X)-1$ if $\Tan(X)=\PP^{2n}$ because through a general point of $T_xX$
there pass $\tau(X)-1$ tangent spaces to $X$ at smooth points  different from $T_xX$.
We isolate the following consequence  for further reference.

\begin{lem}\label{tau(X)=2 implies X rational}
     Let $X^n\subset \PP^{2n}$ be a non-degenerate variety with $\Tan(X)=\PP^{2n}$. If $\tau(X)= 2$, then $X$ is birational
     to a general tangent space to $X$  via the Scorza map. 
\end{lem}

\subsection{Examples of $X^n\subset\PP^{2n}$ with $\tau(X)=2$}

We shall  present several series of examples for every $n\geq 2$.

\begin{exam}\label{QELex} 
We define the  {\it secant defect  of $X$} as $\delta(X)=2n+1-\dim(\Sec(X))\geq 0$. For $X^n\subset \PP^N$ and for a general $p\in \Sec(X)$ 
the {\it entry locus of $X$ with respect to $p$} is: 
$$\Sigma_p=\overline{\{x_1\in X\;:\;\exists\; x_2\in X\setminus x_1\mbox{ with }p\in\langle x_1,x_2\rangle\}}\subseteq X.$$
The generality of $p\in \Sec(X)$ assures that $\Sigma_p\subseteq X$ is equidimensional of dimension $\delta=\delta(X)$. 

Since $\deg(\Sigma_p)\geq 2$, the varieties with the {\it simplest} secant behaviour are those for which $\deg(\Sigma_p)=2$. If $\delta(X)=0$, this means that through
a general point of $\Sec(X)$ there passes a unique secant line to $X$ (the expected behavior). 

A projective variety  $X^n\subset\PP^{N}$ is said to be a {\it $QEL$-variety of type $\delta\geq 0$},  i.e. a variety with quadratic entry locus of type $\delta$, if
$\Sigma_p\subset X$ is a smooth quadric hypersurface of dimension $\delta=\delta(X)$ for $p\in \Sec(X)$ general. 

Here the relevant  remark for us is the following: if $X^n\subset\PP^{2n}$ is a $QEL$-variety of type $\delta=1$, then $\tau(X)=2$. Indeed, 
the tangent lines to $X$ at smooth points passing through $p$ are tangent lines to $\Sigma_p$ and
viceversa. Since $\tau(\Sigma_p)=2$ the claim follows.

Let us give some examples of $QEL$--varieties. For $n\geq 1$ a smooth rational normal scroll $$S(a_1,\ldots,a_n)=\PP(\bigoplus_{i=1}^n\mathcal O_{\PP^1}(a_i))\subset \PP^{2n}$$ embedded via $\mathcal O_{\PP(\mathcal E)}(1)$ and with $a_1+\ldots+a_n=n+1$ is a $QEL$-variety with $\delta=1$. The smoothness assumption implies $a_i\geq 1$ for every $i=1,\ldots,n$.

We shall see in Theorem \ref{class:omega_2=2} below  that the unique smooth surface in $\PP^4$ with $\tau=2$ is $S(1,2)$, up to projective equivalence.
 For $n\geq 3$ there exist smooth $QEL$-varieties $X^n\subset\PP^{2n}$ of type $\delta=1$ which are not scrolls over $\PP^1$. For example, consider $\mathbb G(1,4)\subset\PP^9$, which is a $QEL$-variety of type $\delta=4$. Then a general three-dimensional linear section $X^3\subset\PP^6$ of $\mathbb G(1,4)\subset\PP^9$. Since $\mathbb G(1,4)\subset\PP^9$ is a $QEL$-variety of type $\delta=4$, we deduce that  $X^3\subset\PP^6$  is a smooth $QEL$-variety of type $\delta=1$ with $\Pic(X)\simeq \mathbb Z\langle\mathcal O_X(1)\rangle$ and $\omega_X^*\simeq \mathcal O_X(2)$ (i.e. it is a  prime del Pezzo threefold).

A similar example is a  general five-dimensional linear section $X^5\subset\PP^{10}$ of the spinor variety $S^{10}\subset\PP^{15}$, which is a homogeneous $QEL$-variety of type $\delta=6$. In this case $X^5\subset\PP^{10}$  is a smooth  $QEL$-variety of type $\delta=1$ with $\Pic(X)\simeq \mathbb Z\langle\mathcal O_X(1)\rangle$ and  $\omega_X^*\simeq \mathcal O_X(3)$ (i.e. it is a prime  Mukai fivefold).
\end{exam}

\begin{rem}\label{Johnsontau} If $N\geq n+2$, a general secant line is not trisecant
by the well-known Trisecant Lemma. Thus the projection  $\pi_p:\PP^N\map\PP^{N-1}$ from a general $p\in\PP^N$ induces a $2:1$ morphism from $\Sigma_p$ onto its image. The difficulty in computing $\tau(X)$ for an arbitrary $X^n\subset\PP^{2n}$ is due to the fact that the curve $\Sigma_p$ may be very singular and reducible, preventing the application of  Riemann-Hurwitz formula to 
the covering $\pi_p:\Sigma_p\to \pi_p(\Sigma_p)$ in order to determine the number of smooth ramification  points, which is equal to $\tau(X)$. The problem is similar to the 
computation of the degree of the dual curve of an irreducible  curve $C\subset\PP^2$, which is deeply affected by the complexity of its  singularities (the Pl\" ucker formulas work effectively only if $C$ has mild singularities).

When $X^n\subset\PP^N$ is smooth or at has  a finite number of nodal singularities  at worst, one can  prove that $\tau(X)$ is even by applying Riemann-Hurwitz formula on the 
desingularizations of $\Sigma_p$ and of  $\pi_p(\Sigma_p)$.
\end{rem}

For smooth varieties we have the following result.

\begin{pr}\label{smoothtau2} Let $X^n\subset\PP^{2n}$ be a smooth non-degenerate  variety of dimension $n\geq 1$. Then $\tau(X)=2$ if and only if $X^n\subset\PP^{2n}$ is a $QEL$-manifold of type $\delta=1$.
\end{pr}
\begin{proof}
We have seen above that  $QEL$-manifolds $X^n\subset\PP^{2n}$ of type $\delta=1$  have $\tau(X)=2$. Suppose now $\tau(X)=2$. Let $p\in\PP^{2n}$ be general and let notation be as in Example \ref{QELex}.
The main results in \cite[\S 5.2]{Johnson} yield $\deg(\Sigma_p)=\tau(X)$ since $X\subset\PP^{2n}$ is smooth, a fact which can also be deduced from Severi's Double Point Formula \eqref{Sev:formula}. Then $\Sigma_p\subset X$ is an equidimensional curve of degree two
such that through $p$ there pass two lines which are limits of secant lines to $\Sigma_p$ (and tangent lines to $X$) and thus it is necessarily  a smooth conic.
\end{proof}
\begin{cor}\label{classtau2smooth} Let $X^n\subset\PP^{2n}$ be a smooth  variety with $\tau(X)=2$. Then either $X^n\subset\PP^{2n}$
is projectively equivalent to $S(\underbrace{1,\ldots,1}_{n-1},2)$ or $X^n\subset\PP^{2n}$ is a linearly normal rational Fano manifold with $\Pic(X)\simeq\mathbb Z\langle\mathcal O_X(1)\rangle$
and index $i(X)=(n+1)/2$.

In particular, if $n=2m\geq 2$, then  $X^{2m}\subset\PP^{4m}$ is projectively equivalent to $S(1,\ldots,1,2)$.
If $n=3$, respectively  $5$, then either $X^n\subset\PP^{2n}$ is a rational normal scroll as above or it is one of the two Fano manifolds of dimension 3, respectively 5, described at the end of Example
\ref{QELex}.
\end{cor}
\begin{proof} Since $\tau(X)=2$ and  $X^n\subset\PP^{2n}$, then $X\subset\PP^{2n}$ is linearly normal. The first part follows from the classification of conic-connected
manifolds in \cite{IR} and  the {\it Parity Theorem} for $QEL$-manifolds (see for example \cite[Proposition 4.3.1, part 3)]{Russo.Book2016}), yielding $i(X)=(n+1)/2$. The second part is a consequence of the classification of del Pezzo and  Mukai manifolds.
\end{proof}

We are not aware of the existence of  $QEL$-manifolds $X^{2m+1}\subset\PP^{4m+2}$ of type $\delta=1$ with $m\geq 3$, $\Pic(X)\simeq \langle\mathcal O_X(1)\rangle$ and index $i(X)=m+1$.
So one may ask if  the last part of Corollary \ref{classtau2smooth} describe all smooth varieties $X^n\subset\PP^{2n}$ with $\tau(X)=2$. 

The application of the next result  will provide other interesting examples of (possibly singular) varieties $X^n\subset\PP^{2n}$ with $\tau(X)=2$. First we recall a definition. For an arbitrary variety $X\subseteq\PP^N$, the {\it vertex} of $X$, denoted by $\Ver(X)$, is the set of points $x\in X$ such that for every $x'\in X\setminus x$ we have
$\langle x,x'\rangle\subseteq X$. A variety $X\subseteq\PP^N$ such that $\Ver(X)\neq\emptyset$ is called a cone. Let us recall that $\Ver(X)$ is a linear subspace,
if non-empty.

\begin{lem}\label{omegadegTX} Let $Y^n\subset\PP^{2n+1}$ be a
variety such that $\Tan(Y)\subset\PP^{2n+1}$ is a hypersurface of degree four and such that $\tau(Y)=1$.
If $q\in \Sing(\Tan(Y))\setminus\vvert(\Tan(Y))$ and  if $Y_q\subset \PP^{2n}$ is
the projection of $Y$ from $q$, then $\tau(Y_q)=2$ and $\mult_q(\Tan(Y))=2$.
\end{lem}
\begin{proof} Let $q\in \Sing(\Tan(Y))\setminus\vvert(\Tan(Y))$. The projection  $\pi_q:\PP^{2n+1}\map \PP^{2n}$ restricted to $\Tan(Y)$ induces a generically finite rational map $\Tan(Y)\map \PP^{2n}$ of degree
$4-\mult_q(\Tan(Y))\geq 1$. If $y\in Y$ is general, then $\pi_q(T_yY)=T_{\pi_q(y)}X$, yielding $\Tan(X)=\PP^{2n}$. 
Thus from $\tau(Y)=1$ we deduce $\tau(X)=4-\mult_q(\Tan(Y))$. Since $\mult_q(\Tan(Y))\geq 2$ and since $\tau(X)\geq 2$ by Theorem \ref{If n>=3 then tau(X)>=2},
the claims follow.
\end{proof}

\begin{exam}\label{X333} A projective variety $Y^n\subset\PP^{2n+1}$ is called an $OADP$-variety, i.e. a {\it variety with one apparent double point},
if through a general $p\in\PP^{2n+1}$ there passes a unique secant line to $Y$. The name originates from the fact that $Y_p$ acquires  a unique  {\it double point},
whose tangent cone consists of the union of the projections of the tangent spaces at the two points  in $Y$ spanning the unique secant passing through $p$. More generally, the number of secant lines to $Y$ passing through a general $p\in \PP^{2n+1}$ is the {\it number of apparent double points of $Y$}. Equivalently, a variety $Y^n\subset\PP^{2n+1}$ is an $OAPD$-variety if and only if it is a $QEL$-variety of type $\delta=0$.

For an $OADP$-variety $Y^n\subset\PP^{2n+1}$ we necessarily have $\tau(Y)=1$ by Zariski's Main Theorem. Moreover, we also have $\dim(\Tan(Y))=2n$. Indeed, the projection from a general $T_yY$ onto $\PP^n$ is birational so  that the strict transform of $y$ in $\PP^n$ is a divisor. Thus $\dim(\overline{\varphi_y(\PP(t_yY)}))=n-1$ and we can apply formula \eqref{dimTXssecond}. Therefore $OAPD$-varieties $Y^n\subset\PP^{2n+1}$ with $\deg(\Tan(Y))=4$  produce examples of  varieties $X^n\subset\PP^{2n}$ with $\tau(X)=2$ via Lemma \ref{omegadegTX}.

The only  example of an $OADP$-curve is the twisted cubic $Y\subset\PP^3$. As already known to Cayley, we have $\deg(\Tan(Y))=4$, $\Sing(\Tan(Y))_{\red}=Y$ and $X=Y_q\subset\PP^2$ is a conic for any $q\in Y$.
This elementary fact has been generalized to the class of  projective varieties $Y^n\subset\PP^{2n+1}$ which are 3-covered by twisted cubics, i.e. such that through three general points $y_1,y_2,y_3\in Y$ there passes a twisted cubic contained in $Y$. Such varieties  are $OADP$'s by \cite[Corollary 5.4]{PR1} and have  $\deg(\Tan(Y))=4$ by \cite[Lemma 5.4]{PR3} so that a general internal projection 
gives examples of $X^n\subset\PP^{2n}$ with $\tau(X)=2$ by Lemma \ref{omegadegTX}.

The latter class of varieties includes: smooth rational normal scrolls $Y^n\subset\PP^{2n+1}$; the Segre embedding of $Y^n=\PP^1\times Q^{n-1}\subset\PP^1\times \PP^n$, $n\geq 2$, where $Q^{n-1}\subset\PP^n$ is an irreducible quadric hypersurface; four homogeneous varieties $Y^n$ with $n=3, 5, 9, 27$ and many other singular examples produced
as {\it twisted cubics over a cubic Jordan algebra} (see \cite{PR1, PR2, PR3, Russo.Book2016}).

Let us remark that the projection of a variety as above from a point $q\in Y$ is a variety $X^n\subset\PP^{2n}$
such that through two general points there passes a smooth conic and, in fact, it is a  $QEL$-variety of type $\delta=1$ (although $X$ is singular in most cases).
\end{exam}

We shall now outline a construction due to Verra that  produces infinitely many examples of $OADP$-varieties $Y^n\subset\PP^{2n+1}$ such that $\deg(\Tan(Y))=4$ starting from
an $OADP$-variety $Z^m\subset\PP^{2m+1}$, $m<n$, with $\deg(\Tan(Z))=4$. In this way, for fixed $n\geq 2$, one can construct  examples of singular $OADP$-varieties of arbitrary large degree.

\begin{exam}\label{Verra}
Let  $Z\subset {\PP}^{2n+1}$ be a degenerate $OADP$-variety of dimension $m<n$,
which spans a linear space $V$ of dimension $2m+1$. 
Take a linear space $W\subset \PP^ {2n+1}$ 
of dimension $2(n-m)-1$ such that $V\cap W=\emptyset$.
Let 
$$S(W,Z)=\overline{\bigcup_{w\in W, z\in Z,\\ w\neq z}\langle w, z\rangle}$$ indicate the {\it join of $W$ and $Z,$} which is therefore  a cone
over $Z$ with vertex $W$. Let $Y\subset S(W,Z)$ be a non-degenerate, not secant-defective variety of dimension $n$ which intersects the general ruling $$\Pi_z:=\langle W,z\rangle\cong \PP^ {2(n-m)}$$ of $S(W,Z)$ 
along a linear subspace $P_z$ of dimension $n-m$.
This implies that the projection with center $W$, $\pi_W:\PP^{2n+1}\map V$, restricts to a dominant rational map
$\pi:Y\map Z$. If $P_{z_i}=\langle W,z_i\rangle$, $1\le i\le 2$, are the closures of two general fibers of $\pi$, then
$P_{z_1}\cap P_{z_2}=\emptyset$.

Indeed, the first claim is  clear, while the second  follows by Terracini's Lemma because $Y$  is not secant-defective and so
two general tangent spaces do not intersect. 
These varieties are called  {\it Verra varieties}.

There is an alternative definition of these varieties.  Let $Z$ be as above. 
Consider a rational map $f: Z\map \mathbb G(n-m-1, W)$ such that,
if $z_1, z_2\in Z$ are general points, then $f(z_1), f(z_2)$ are skew 
linear subspaces of $W$. Set
\[
Y:=\overline { \bigcup_{x\in Z_0}\langle x,f(x)\rangle }
\]
where $Z_0\subset Z$ is the open subset where $f$ is defined and the above property holds. Then $Y\subset\PP^{2n+1}$ is a Verra
variety.

 Verra varieties are $OADP$'s. Indeed,   let $p\in \PP^ {2n+1}$ be a general point so that $q=\pi_W(p)$
is a general point of $V$. A secant line to $Y$ through $p$ is a general secant line to
$Y$ and it projects to a general secant line to $Z$ passing through $q$. Since $Z\subset V$ is an $OADP$,
there is only one such secant line $L$ intersecting $Z$ at two points $z_1,z_2$. 
Hence all secant lines to $Y\subset\PP^{2n+1}$ 
through $p$ lie in the linear space $T=\langle L, W\rangle=S(L,W)$ of dimension $2(n-m)+1$.
Also $T$ intersects $Y$ along the two linear spaces $P_i\subset \langle z_i, W\rangle$, 
$1\le i\le 2$, of dimension $n-m$,  whose union spans
$T$. The assertion follows because there is only one secant line to $P_1\cup P_2$ passing through
$p\in T$. 

Similarly, we get  $\Tan(Y)=S(W,\Tan(Z))\subset\PP^{2n+1}.$ Therefore, if $Z^m\subset\PP^{2m+1}$ is an $OADP$-variety with
$\deg(\Tan(Z))=4$, then we can produce an infinite series of arbitrary large degree $OADP$-varieties $Y^n\subset \PP^{2n+1}$
with $\deg(\Tan(Y))=4$. Via Lemma \ref{omegadegTX} we thus obtain examples  of varieties $X^n\subset\PP^{2n}$
with $\tau(X)=2$ and having arbitrary large degree.

Starting with $Z\subset\PP^5$ a twisted cubic, we consider the cone $C_Z=S(W,Z)\subset\PP^5$ with vertex a line $W$ skew to
$V=\langle Z\rangle$
and construct Verra surfaces $Y_d\subset C_Z\subset\PP^5$ of degree $d\geq 4$ with $\Tan(Y_d)=\Tan(C_Z)=S(W,\Tan(Z))\subset\PP^5$
a degree four hypersurface. 

The first example is $Y_4=S(1,3)\subset\PP^5$, which is smooth and  is obtained by fixing an isomorphism $f:Z\to W=\PP^1$. If we project it from a point $q\in Y_4\setminus W$,
we get a smooth cubic rational scroll $S(1,2)\subset\PP^4$. If we project it from a point  $q\in C_Z\setminus (W\cup Y_d)$ we get a singular
rational scroll of degree four. 

The rational scrolls $Y_d\subset\PP^5$ obtained via  morphisms $f:Z\to W=\PP^1$ of degree $d-3\geq 2$ are singular along the line $W\subset Y_d$  and
produce singular rational scrolls of degree $d$ if $q\in C_Z\setminus (W\cup Y_d)$, respectively of degree $d-1$ if
$q\in Y_d\setminus W$.
\end{exam}
\subsection{One dimensional developable families of linear spaces}\label{dev:fam}
We shall now generalize the contents of Example \ref{dev:surf} and provide various applications.

Let $C\subset \mathbb G(r,N)$ be a smooth curve with $r\geq 1$ and
let  $\mathcal Y=\PP({\mathcal{S}}_{|C})$ with $\mathcal{S}$ the universal locally free sheaf of rank $r+1$ over $\GG(r,N)$. We get a  diagram 
\begin{equation}\label{Tkuniv}
\begin{tikzcd}
	{\mathcal Y} & {\PP^N} \\
	{C}
	\arrow["q", from=1-1, to=1-2]
	\arrow["p"', from=1-1, to=2-1]
\end{tikzcd},
\end{equation}
where $q$ the tautological morphism.  We call \eqref{Tkuniv} a non-degenerate one dimensional family of linear spaces in $\PP^N$ if  $Y=\overline{q(\mathcal Y)}\subseteq\PP^N$
has dimension $r+1$ . For $t\in C$, let $$\PP^r_t=q(p^\inv(t))\subset Y.$$ 

A non-degenerate one-dimensional family \eqref{Tkuniv} is called {\it developable} if, for $t\in C$ general and for $y\in \PP^r_t\cap Y_{\reg}$ general, the tangent space $T_yY$ does not depend on $y$ as above. The image of the Gauss map of a non-degenerate developable $Y$  is biholomorphic to $C$ because  $\PP^r_t$ is  the closure of the fiber of the Gauss map of $Y$ passing through a general point of $\PP^r_t$ for $t\in C$ general. In other words, we can say that the non-degenerate one-dimensional family $Y$ is developable if and only if $Y$ is a developable scroll over a curve, the scroll structure being given by the Gauss map. 

Reasoning as in Example \ref{dev:surf} it is immediate to see that if 

\begin{equation}\label{parY}
\mathbf x(t)=\lambda_0a_0(t)+\ldots+\lambda_ra_r(t)
\end{equation}
is a local parameterization  of $Y$, where $\lambda_i\in \mathbb C$ and
$a_i:\Delta\to \mathbb C^{N+1}$ with $\Delta\subset \mathbb C$ a small disk,
then the non-degenerate one-dimensional family is developable if and only if 
$$\dim(\mathcal L(a_0(t),\ldots, a_r(t), a_0^{\prime}(t),\ldots, a_r^{\prime}(t)))= r+2.$$

Let us recall that, letting $\PP^r_t=\mathbb P(\mathcal L(a_0(t),\ldots, a_r(t)))$,  the intersection of the {\it infinitely near space of $\PP^r_t$} defined by $C$, let us say $(\PP^r_t)^\prime$, with $\PP^r_t$ is by definition the projectivization of the kernel of the homomorphism:
$$\frac{\partial}{\partial t}:\mathcal L(a_0(t),\ldots, a_r(t))\to \frac{\CC^{N+1}}{\mathcal L(a_0(t),\ldots, a_r(t))},$$
defined by
$$\frac{\partial}{\partial t}(\lambda_0a_0(t)+\ldots+\lambda_ra_r(t))=[\frac{\partial}{\partial t}(\lambda_0a_0(t)+\ldots+\lambda_ra_r(t))]=[\lambda_0a_0^\prime(t)+\ldots+\lambda_ra_r^\prime(t)].$$
In conclusion,   the one-dimensional family \eqref{parY} is developable if and only if 
 a linear space in the family and the infinitely near space intersect in a $\PP^{r-1}_t,$
 consisting of foci of the family $\mathcal Y\to C$ of $\PP^r$'s  and also contained in $\Sing(Y)$ (see also \cite[$\S$ 2]{Griffiths.Harris.1979} and 
  \cite[$\S$ 2.2.4]{Fischer.Piontkowski.Book2001}). 
  
   More precisely, if the infinitely near $\PP^{r-1}_t$ does not depend on $t$, then  $Y\subset\PP^N$ is
  a cone over a curve with vertex $P=\PP^{r-1}_t$. If $Y\subset \PP^N$, then the infinitely near $\PP^{r-1}_t$'s describe an irreducible component
  of maximal dimension $r$ of the singular locus of $Y$.

If $Y\subset\PP^N$ is a developable scroll over a curve $C$ and if $L=\Ver(Y)=\PP^s$ with $0\leq s\leq r-2$, then $Y=S(L,Z)$ where $Z\subset \PP^{N-s-1}$ is a developable scroll over the same curve $C$ of dimension at least two and not a cone. Thus we shall concentrate on  developable scrolls over a curve $Y\subset\PP^N$ of dimension $r+1\geq 2$ that are not  cones. By \cite[Lemma (2.2)]{Griffiths.Harris.1979}, there exists a parametrization \eqref{parY} with $a_i(t)=a_0^{(i)}(t)$
for every $i=1,\ldots, r$, showing  that $Y\subset \PP^N$ is the variety of osculating $r$-spaces $\PP(\mathcal L(a_0(t), a_0^{(1)}(t),\ldots,a_0^{(r)}(t)))$ to the  curve $\PP(a_0(t))\subset\PP^N$ and that $\PP^{r-1}_t=\PP^r_t\cap(\PP^{r}_t)^\prime$ is the hyperplane of equation $\lambda_r=0$.

Under our hypothesis, the  one-dimensional family   $\PP^{r-1}_t=\PP^r_t\cap(\PP^r_t)^\prime$, $t\in U\subseteq C$, which is contained in the singular locus of $Y$ and which  is the variety of osculating $(r-1)$-spaces to the curve $\PP(a_0(t)))$, is developable with parametrization 
$$\mathbf x(t)=\lambda_0a_0(t)+\ldots+\lambda_{r-1}a_0^{(r-1)}(t).$$
Indeed, also in this case $\PP^{r-1}_t$  and  $(\PP^{r-1}_t)^\prime$ intersect in a hyperplane 
$\PP^{r-2}_t\subset \PP^r_t$ of equation $\lambda_{r-1}=0$ because 
$$[\frac{\partial}{\partial t}(\sum_{i=0}^{r-1}\lambda_ia_0^{(i)}(t))]=[\lambda_{r-1}a_0^{(r)}(t)].$$

Let 
$$Y^s=\bigcup_{t\in U} \PP^{r-s}_t.$$ Then $Y^0=Y$ and for $s\geq 1$ we have 
 $Y^{s}\subseteq \Sing(Y^{s-1})$, i.e. the singular locus of $Y^{s-1}$ contains  the variety of $(r-s)$-osculating spaces to the curve $\overline{\PP(a_0(t))}$. 
 \medskip

Considering a family of linear spaces as in \eqref{fam:lin:n:r} with $V\subseteq \GG(r,N)$ and a smooth curve $C\subset V$, we get
a one--dimensional family $p^\inv(C)\to C$. In general, this family is not developable and the existence of developable families passing through a general point of $V$ is a very
strong condition.

To wit, we can combine the above definitions and the general facts about foci to deduce the following important remark by Italiani: if $p:\mathcal V\to V$ is  a non-degenerate 
two--dimensional family of planes in $\PP^4$ (a non-degenerate {\it congruence of planes in $\PP^4$}), then  through a general point of $V$ there passes a one-dimensional developable family of planes if and only if 
the general focal conic is reducible
(see \cite[\S 2]{Italiani} and also the more recent treatment in \cite{Pedreira}).
\subsection{Classification of irreducible surfaces in $\PP^4$ with $\tau=2$}\label{omega_2=2}

The aim of this subsection is to prove the following classification result.

\begin{thm}\label{class:omega_2=2} Let $X^2\subset\PP^4$ be an irreducible surface such that $\tau(X)=2$.
Then $X$ is a scroll over a rational curve.

Moreover, if  $X$ is also smooth, then it is projectively equivalent to $S(1,2)\subset\PP^4$.
\end{thm}

To this aim we shall briefly recall the theory of second order foci in our particular setting,
as outlined by C. Segre in \cite{Segre2}. Segre  applied this theory  to study congruences of planes in $\PP^4$ of order two, i.e.
two-dimensional families of planes in $\PP^4$ as  in \eqref{fam:lin:n:r}  such that $\deg(q)=2$ (see the final pages of \cite{Segre2}). 
The first remark by Segre is completely general. Suppose we have a diagram
as in \eqref{fam:lin:n:r} and  there exists a smooth point  $q(r)\in q(\mathcal V)$ such that $\length(q^\inv(q(r)))>\deg(q)$. Then $r$ (or $q(r)$)
is a fundamental point for the family of linear spaces (see  \cite[Lemma 1.7]{Chiantini.Ciliberto.korder} for a proof  similar
to Segre's original argument or use the general fact that a morphism
with finite fibers  onto a non-singular variety is flat). 

Segre's idea for studying families of linear spaces as in \eqref{fam:lin:n:r}, with $N=n+r$ and $\deg(q)=2$,  or subvarieties $X^n\subset\PP^{2n}$
with $\tau(X)=2$,  is to use foci $r\in \mathcal V$  to produce
a subscheme of $q^\inv(q(r))$ supported at $r$ and of length at least  three by means of the differential of the map $q$ at $r$
(for a focus we have  $\length_r q^\inv(q(r))\geq 2$ by definition). 

So, following Segre, one looks for a {\it second-order infinitely near space at r}  as a {\it focus of the family of
foci}, i.e.  an infinitely near variety of foci to the variety of foci. If this occurs, we have finally constructed {\it  a second-order focus} (see \cite{Segre2}).

\begin{exam}\label{2foci:dev}Let us see some examples that illustrate 
how the  definition of higher-order foci should be introduced under suitable hypothesis like smoothness of
the focal loci or of its support, etc. If $F_v\subset \PP^r_v$ is the focal locus of the family \eqref{fam:lin:n:r}, we shall indicate by $F^2_v\subseteq F_v$
the locus of second-order foci of the family \eqref{fam:lin:n:r}, i.e. the locus where $q$ ramifies with order at least three in some direction.

 Let $Y^n=T^{(n-1)}C\subset \PP^N$, $n\geq 2$ and $N\geq n+1$, be the $(n-1)$-osculating scroll to an irreducible non-degenerate
curve $C\subset\PP^N$. The foci of the family of $(n-1)$-osculating spaces on a general osculating space $T^{(n-1)}_pC$ is the 
hyperplane $T^{(n-2)}_pC.$ If $n\geq 3$, the foci of the family of $(n-2)$-osculating spaces on a general osculating space $T^{(n-2)}_pC$ is the 
hyperplane $T^{(n-3)}_pC$ and so on.
Thus, if $n\geq s+1$,  $T_p^{(n-s)}C$, $s\geq 2$, consists of foci of the $(s-1)$-order of the family of linear spaces $T_p^{(n-1)}C$'s. In particular, for $p\in C$ general and for $q(r)\in T^{(n-s)}_pC\setminus T^{(n-s-1)}_pC$ general, we have $\length_rq^\inv(q(r))\geq s$.

Let $Y^3=S(p,TC)\subset \PP^N$ with $C\subset \PP^{N-1}$ an irreducible non-degenerate curve and $p\not\in\PP^{N-1}$. Then the  family of planes $\langle p, T_rC\rangle$ is developable. The focal
line in the plane $\langle p, T_rC\rangle$, $r\in C$,  is $\langle p, r\rangle$. Then  $r$ is a second-order focus of the family of planes and $p$ is a fundamental point. In particular, $C$ is again contained in the locus of second-order foci.

Let $Y^3=S(L,C)\subset\PP^N$ with $L\subset\PP^N$ a line and $C\subset\PP^{N-2}$ an irreducible curve with $\PP^{N-2}\cap L=\emptyset$. The family of planes $\langle L,r\rangle$, $r\in C$,  is developable and $L$ is a locus of fundamental points. 
\end{exam}

We now concentrate on the case of the family of tangent planes to  a surface in $\PP^4$.
Let us recall that in this case, for $x\in X$ general,  the  focal conic  $C_x\subset T_xX$ is either the union of two distinct lines $L_x^i\subset T_xX$ passing through $x$ or it is supported
on a line $L_x\subset T_xX$ through $x$.

\begin{lem}\label{lem:F2} Let $X\subset\PP^4$ be a non-degenerate irreducible surface such that $\Tan(X)=\PP^4$. If $L_x^i\subseteq F^2_x$
for some $i$ and for $x\in X$ general, then $X\subset \PP^4$ is a scroll over a curve.
 \end{lem}
\begin{proof} Let $F_i\subset\PP^4$ be the image via $q$ of the locus $\mathcal F_i\subset T_X$ given by the closure of the lines $L_x^i\subseteq F^2_x$, $x\in X$ general. Under our hypothesis, the morphism  $q_{|\mathcal F_i}:\mathcal F_i\to F_i$
ramifies everywhere, yielding $\dim(F_i)\leq 2$. From $X\subseteq F_i$ we deduce that $X$ is an irreducible component of $F_i$. Then
$$X\subseteq \overline{\bigcup_{x\in U\subseteq X}L_x^i}\subseteq F_i$$ 
so that equality holds in the first inclusion. Since $\dim(\Tan(X))=4$ implies $\dim(|\SSx|)=1$, we deduce that through a general $x\in X$ there passes a unique line $L_x^i$, concluding the proof.
\end{proof}

\begin{lem}\label{lem:FF2} Let $X\subset\PP^4$ be a non-degenerate irreducible surface such that $\Tan(X)=\PP^4$. Then $X\subset\PP^4$ is contained in the locus of second order foci of the family of tangent spaces.
\end{lem}
\begin{proof} Let us fix a general $x\in X$ and let $C_x=L_x^1\cup L_x^2$ or $(C_x)_{\red}=L_x$ be the focal conic in $T_xX$. Let $v_i=\PP(t_xL_x^i)\in \PP(t_xX)$, respectively $v=\PP(t_xL_x)$, be the directions of the corresponding lines. In the first case the points $v_i$'s are conjugate with respect to $\SSx$, while in the second case $v\in \Bs(\SSx)$.

In the first case, we can separate analytically the two directions in a neighbourhood of $x$ and define via the family of tangent lines $t_yL_y^1$ a flow in an analytic  neighbourhood $V$ of $x$. Let $D_x\subset  V$ be the analytic curve passing through $x$ and  having the previous prescribed tangent directions at every point $y\in U$. In the second case we simply take the family of tangent lines $t_yL_y$ to define the flow. The curve $D_x\subset V$ defines the one-dimensional family of tangent planes
$$Y_{D_x}=\bigcup_{p\in D_x}T_pX\subset \Tan(X)=\PP^4.$$
By construction $Y_{D_x}\subset \PP^4$ is developable because at each  $p\in D_x$ the infinitesimal near plane $(T_pX)^\prime$ cuts $T_pX$ along $L^1_p$, respectively $L_p$.

If $Y_{D_x}=S(L, \widetilde C_x)$ with $L$ a line and $\widetilde C_x$ an irreducible curve, then $L=L_x^i$, respectively $L=L_x$, is a locus of fundamental points. 
If $Y_{D_x}=S(p,T\widetilde C_x)$ or if $Y_{D_x}=T^{(2)}\widetilde C_x$, then clearly $\widetilde C_x=D_x$ so that  $D_x$ is a locus of second-order foci but $L^1_x$
(or $L_x$) is not. Anyway, a general $x\in X$  is a second order focus, as claimed.
\end{proof}

We are now ready to prove Theorem \ref{class:omega_2=2}  (see also  \cite[pg. 71]{Segre2} for a similar count of parameters for congruences of planes in $\PP^4$ of order two).
\medskip 
\begin{proof}(of Theorem \ref{class:omega_2=2}). By Lemma \ref{lem:FF2} there exists an irreducible component $\widetilde F^2$ of the locus of second order foci $F^2\subset\PP^4$ of the family of tangent planes to $X$ such that $X\subseteq \widetilde F^2$. If $\widetilde F^2$ is an irreducible component of the focal locus $F$, then $\dim(\widetilde F^2)=2$ because every focus in this irreducible component is a second order focus. If not, we have $\dim(\widetilde F^2)=2$. Anyway we have  $\widetilde F^2=X$. 

Since $\tau(X)= 2$, the locus of second-order foci  coincides with  the locus of fundamental points of the family of tangent spaces. Let $G\subseteq \mathcal F^2$ be an irreducible component dominating $X=\widetilde F^2$ via $q:T_X\to\PP^4$. Since every $x\in X$ is fundamental, we have that $\dim(G)=3$ and that $G$ dominates $X$ via $p:T_X\to X$. 
The general fiber of $q_{|G}:G\rightarrow X$, denoted by $G_x$, has pure dimension one and is contained in the focal locus. Hence, for
$x\in X$ general, there exists at least one irreducible component of the focal conic $C_x$ contained in $G_x$. So  we can apply Lemma \ref{lem:F2}  to deduce that $X\subset\PP^4$ is a scroll over a curve, which is rational by Lemma \ref{tau(X)=2 implies X rational}.

Finally observe that a smooth scroll  $X\subset\PP^4$ is linearly normal  and, being rational, it is necessarily  isomorphic to $S(1,2)$.
\end{proof}

\begin{rem}\label{rem:n3:tau=2} If $n\geq 3$, the previous count of parameters does not work and  we constructed examples of smooth irreducible varieties $X^n\subset\PP^{2n}$ with $\tau(X)=2$ that are not scrolls over a rational curve (see Examples \ref{QELex}, \ref{X333} and \ref{Verra}).
\end{rem}

\section{A lower bound for the degree of the  tangent variety when $\Tan(X)\subsetneq\Sec(X)$}\label{Sec:tandegree}

Suppose that $X^n\subset\PP^N$ is a  non-degenerate variety such that $\Tan(X)\subsetneq \PP^N$. Let us recall that if $X$ is also smooth, then  either $\Tan(X)=\Sec(X)$ or $\dim(\Tan(X))=2n$ and $\dim(\Sec(X))=2n+1$ by Zak's Theorem on tangencies  (see for example \cite{Russo.Book2016}).  We are looking for an effective  lower bound on the degree of $\Tan(X)\subsetneq\PP^N$ and then  try to classify varieties attaining the bound in the same spirit of the  results  in \cite{Ciliberto.Russo.2006} for $\Sec(X)$. If  $\Tan(X)=\Sec(X)\subsetneq\PP^N$, then a lower bound on the degree follows from the results in  {\it loc.~cit.}

\begin{thm}[{\cite[Theorem 4.2]{Ciliberto.Russo.2006}}]\label{Lower bound on deg(SX))} Let $X^n\subset \PP^N$ be a non-degenerate variety such that  $\Sec(X)\subsetneq \PP^N$. Then $$\deg(\Sec(X))\geq \binom{\codim_{\PP^N} (\Sec(X)) +2}{2}.$$
\end{thm}

The following result is elementary but it will be the starting point for our induction argument leading to an optimal  lower bound for $\deg(\Tan(X))$ in terms of its codimension.

\begin{pr}\label{deg(TX)=3 implies TX=SX is a hypersurface} Let $X^n\subset\PP^N$ be a non-degenerate variety with  $\Tan(X)\subsetneq \PP^N$.
Then $deg(\Tan(X))\geq 3$. Moreover, if $\deg(\Tan(X))=3$, then  $\Tan(X)=\Sec(X)$ is a hypersurface in $\PP^N$. 

In particular,   if $\Tan(X)\subsetneq \Sec(X)$, then $\deg(\Tan(X))\geq 4$.
\end{pr}
\begin{proof}
    Since $X$ is non-degenerate and since $X\subseteq \Sing(\Tan(X))\subsetneq \Tan(X)\subsetneq \PP^N$,  we have $\deg(\Tan(X))\geq 3$ because the singular locus of a quadric hypersurface is a linear subspace. Suppose $\deg(\Tan(X))= 3$. Since $3=\deg(\Tan(X))\geq \codim_{\PP^N}(\Tan(X))+1$, we have that either $\Tan(X)$ is a cubic hypersurface in $\PP^N$ or $\Tan(X)$ is a variety of minimal degree of codimension $2$ in $\PP^N,$ that is  a  cubic rational normal scroll. Since the singular locus of a rational normal scroll is a linear space, the second case cannot  occur. In the former case, we claim that $\Sec(X)\subseteq \Tan(X)$, yielding $\Tan(X)=\Sec(X)$. Indeed. If $L$ is a general secant line to $X$ then $L\cap X=\{p_1,p_2\}$ with $p_1, p_2\in X$. Since $p_i\in\Sing(\Tan(X))$, we have $\mult_{p_i}(L\cap \Tan(X))\geq 2$ so  Bézout's Theorem yields $L\subset \Tan(X)$, proving the claim.
\end{proof}

From now on we shall concentrate exclusively on the case $\Tan(X)\subsetneq \Sec(X)\subseteq\PP^N$.

\begin{rem}\label{Lemma on projections of tangent varieties}
Given a variety $Y\subset \PP^N$ and a point $p\in \PP^N$, we denote by $Y_p \subseteq\PP^{N-1}$ the projection of $Y$ from $p$. More generally, if $L\subset \PP^N$ is a linear subspace such that $X\not\subseteq L$, we denote the projection of $X$ from $L$ by $X_L$.  The following are well-known general facts. 

If $p\not\in\Ver(Y)$, then $\dim(Y_p)=\dim(Y)$. If 
$p\in \PP^N\setminus\Ver(Y)$ and if  $$\pi_{p|Y}:Y\map Y_p\subset\PP^{N-1}$$  is the projection from $p$ restricted to $Y$, then
\begin{equation}\label{degproj}
\deg(Y)=\deg(\pi_{p|Y})\cdot \deg(Y_p)+\mult_pY.
\end{equation} 
For an arbitrary $p\in \PP^N$ we have $\Tan(Y)_p=\Tan(Y_p)$ and $\Sec(Y)_p=\Sec(Y_p)$.

In particular, if $X\subset\PP^N$ is a variety such that $\Tan(X)\subsetneq\Sec(X)\subsetneq \PP^N$ and if $p\not\in\Ver(\Tan(X))\cup\Ver(\Sec(X))$, then $\Tan(X_p)\subsetneq \Sec(X_p)$. If $\Tan(X)\subsetneq\Sec(X)=\PP^N$, if $\codim_{\PP^N}(\Tan(X))\geq 2$ and if $p\not\in\Ver(\Tan(X))$, then $\Tan(X_p)\subsetneq \PP^{N-1}=\Sec(X_p)$. So if $\codim_{\PP^N}(\Tan(X))\geq 2$ and if $\Tan(X)\subsetneq \Sec(X)$, then $\Tan(X_p)\subsetneq \Sec(X_p)$ for $p\in X$ general. 

Finally, let us remark that if $X^n\subset\PP^{2n+1}$ is a non-degenerate variety such that $\Tan(X)\subset\PP^{2n+1}$ is a hypersurface of degree four (the minimal possible degree  if $\Tan(X)\subsetneq \Sec(X)$), then 
$\tau(X_p)=(4-\mult_p\Tan(X))\cdot \tau(X)$ as we already remarked in the proof of Lemma \ref{omegadegTX}. Under the previous  hypothesis  $\tau(X)=1$ implies $\mult_p\Tan(X)=2$ for $p\in X$ general, see {\it loc. cit.}
\end{rem}

We are finally ready to provide an effective lower bound on $\deg(\Tan(X))$ when $\Tan(X)\subsetneq \Sec(X)$ and to generalise some of the previous remarks. 

\begin{thm}\label{Lower bound on deg(TX) when TX is not SX)}
 Let $X^n\subset\PP^N$ be a non-degenerate variety such that $\Tan(X)\subsetneq \Sec(X)$ and let $p\in X$ be a general point. 
 Then $$\deg(\Tan(X))\geq 2(N-\dim(\Tan(X))+1).$$ 
 
 Moreover, if equality holds and if $N-\dim(\Tan(X))\geq 2$,  then $\mult_p\Tan(X)=2$, $\deg(\pi_{p|\Tan(X)})=1$
 and $\tau(X)=\tau(X_p)$.
 \end{thm}

\begin{proof}
If $c=N-\dim(\Tan(X))=1$, then $2(N-\dim(\Tan(X))+1)=4$. Proposition \ref{deg(TX)=3 implies TX=SX is a hypersurface}
 assures $\deg(\Tan(X))\geq 4$ under our hypothesis, proving the assertion in this case. From now on we  suppose $c\geq 2$.
By Remark  \ref{Lemma on projections of tangent varieties} we can proceed by induction on $c\geq 2$ and apply  the  induction hypothesis 
to $X_p\subset\PP^{N-1}$ with $p\in X$ general, yielding
 $$\deg(\Tan(X))=\deg(\pi_{p|\Tan(X)})\cdot \deg(\Tan(X_p))+\mult_p\Tan(X)\geq 2(c-1+1)+2=2(c+1)$$
 and concluding the proof of the first part.
 
 If $c\geq 2$ and if $\deg(\Tan(X))=2(c+1)$, then $\mult_p\Tan(X)=2$ and $\deg(\pi_{p|\Tan(X)})=1$ as claimed.
 Since $\deg(\pi_{p|\Tan(X)})=1,$ we have  $\tau(X)=\tau(X_p)$.
\end{proof}

\begin{rem}\label{lowTan} The result is sharp. Indeed, the quadratic Veronese surface $\nu_2(\PP^2)\subset \PP^5$ and the Segre fourfold $\PP^2\times\PP^2\subset \PP^8$  are examples of smooth  varieties for which $\Tan(X)=\Sec(X)\subset\PP^N$ is a cubic hypersurfaces.
\end{rem}

We now show that there exist examples of (smooth) irreducible varieties $X^n\subset\PP^N$  such that $\Tan(X)\subsetneq \Sec(X)$ and with $\deg(\Tan(X))=2(N-\dim(\Tan(X))+1)$. If the varieties are smooth, then necessarily $\dim(\Tan(X))=2n$ and $N\geq 2n+1$. 

\begin{exam}\label{ex:minimalTan} Let $$S(a_1,\ldots, a_n)\subset\PP^{N},$$ $N\geq 2n+1$, be a smooth rational normal scroll. Recall that $a_i>0$ for every $i$,   $\sum_{i=1}^na_i=N-n+1$
and  $\deg(S(a_1,\ldots, a_n))=N-n+1$. 
Moreover, we have $$\dim(\Sec(S(a_1,\ldots,a_n)))=2n+1$$ and 
\begin{equation}\label{degSa1an}
\deg(\Sec(S(a_1,\ldots, a_n)))= \binom{\codim_{\PP^N} (\Sec(S(a_1,\ldots, a_n))) +2}{2}=\binom{N-2n+1}{2}
\end{equation}
(see for example \cite[\S 5]{Ciliberto.Russo.2006}).

Then $\dim(\Tan(S(a_1,\ldots,a_n)))=2n$ and we claim that 
$$\deg(\Tan(S(a_1,\ldots, a_n)))=2(N-2n+1).$$

We shall use the following elementary remark. A linear system of hypersurfaces of degree $d$ on $\PP^N$ of dimension $M$ such that the Koszul syzygies of a base  are generated by linear ones defines a rational map $\phi:\PP^N\map\PP^s$ with linear fibers, i.e. the closure of each  fiber of $\phi$ is a linear subspace of $\PP^N$ (see \cite{Vermeire}). 

This implies that  a variety $X^n\subset\PP^N$ defined by quadratic equations whose first syzygies are generated by  linear ones is a $QEL$-variety (the associated map $\phi$ contracts the secant and tangent lines passing through a point of $\Sec(X)$). In particular, such a variety with $\dim(\Sec(X))=2n+1$ has the property that through a general point of $\Sec(X)$, respectively of $\Tan(X)$, there passes a unique secant line to $X$ (i.e. $\mu(X)=1$), respectively a unique tangent line to $X$ (i.e. $\tau(X)=1$). Therefore $\sigma(X)=\deg(\Sec(X))$ and $\omega_n=\deg(\Tan(X))$ for this class of varieties.

 We apply Severi's Double Point Formula \eqref{Sev:formula} and \eqref{degSa1an} to
deduce:
\begin{equation}\label{firstsum}
(N-2n+1)\cdot (N-2n)=2\cdot \deg(\Sec(S(a_1,\ldots, a_n))=(N-n+1)(N-n)-\sum_{i=1}^n\omega_i.
\end{equation}
If $n=1$ and if $S(a_1)\subset\PP^{a_1}$ is a rational normal curve of degree $a_1$, then 
$$\deg(\Tan(S(a_1)))=\omega_1(S(a_1))=a_1\cdot (a_1-1)-2\cdot\deg(S(a_1))=2\cdot a_1-2=2(a_1-2+1)$$
by the previous formula (or directly using \eqref{PXext}: $\deg(\Tan(S(a_1)))=\deg(c_1(\mathcal P_{S(a_1)}))=2d-2$).  Since a general linear section of a smooth rational normal scroll is a rational normal scroll of the same degree 
and since rational normal scrolls have homogeneous ideal generated by quadratic polynomials with a linear resolution,
we can proceed by induction and, letting $X^{n-1}$ be a linear section of $X= S(a_1,\ldots, a_n)$, show that for every $i=1,\ldots, n-1$:
$$\omega_i(S(a_1,\ldots, a_n))=\omega_{n-i}(X^{n-i})=2(N-i-2(n-i)+1)=2(N-2n+1)+2\cdot i,$$

and so
$$\sum_{i=1}^{n-1}\omega_i=2\cdot (n-1)\cdot (N-2n+1)+n(n-1).$$
Since $$(N-n+1)(N-n)-(N-2n+1)(N-2n)=2n(N-2n)+n(n+1)=2n(N-2n+1)+n(n-1),$$
we get
$$\deg(\Tan(S(a_1,\ldots, a_n)))=2n(N-2n+1)+n(n-1)-2(n-1)(N-2n+1)-n(n-1)$$
$$=2(N-2n+1),$$
as claimed. 
\medskip

On the contrary, since $\tau(S(a_1,\ldots, a_n))=1=\mu(S(a_1,\ldots, a_n))$ for any smooth rational normal scroll as above, one can
compute $\omega_i(S(a_1,\ldots, a_n))$ via Chern/Segre classes of $\mathcal P_{S(a_1,\ldots, a_n)}$, which do not depend
on the $a_i$'s but only on $a_1+\ldots+a_n=N-n+1$. Let $S_{n-i}=S(b_1,\ldots, b_{n-i})\subset \PP^{N-i}$, $i=0,\ldots, n-1$ with  $b_j>0$ and 
with $b_1+\ldots+b_{n-i}=N-n+1$.
Then, for every $i=0,\ldots, n-1$,
we have 
$$\omega_{n-i}(S_{n-i})=\deg(\mathcal O_{\mathcal P_{S_{n-i}}}(1)^{2(n-i)})=2(N-n+1-n+i).$$
Using Severi's Double Point Formula \eqref{Sev:formula} we have
$$2\deg(\Sec(S(a_1,\ldots, a_n)))=(N-n+1)(N-n)-\sum_{i=1}^n\omega_i(S(a_1,\ldots, a_n)),$$
getting the formula in \eqref{degSa1an} for $a_i>0$ and $a_1+\ldots+a_n=N-n$ by going backwards
in the previous computations.

Suppose now $m\geq 1$ and $a_i>0$ as above and consider

$$X_{a_1,\ldots, a_n,m}:=S(a_1,\ldots, a_n,\underbrace{0,\ldots,0}_{m})\subset\PP^{N+m}.$$
This is a singular rational normal scroll of dimension $n+m$ with $$\Ver(X^{n+m})=\Sing(X^{n+m})=L=\PP^{m-1}\subset \PP^{n+m}$$ and having degree
$\deg(X_{a_1,\ldots,a_n,m})=N-n+1=N+m-(n+m)+1$. Then: $$X_{a_1,\ldots, a_n,m}=S(L,S(a_1,\ldots, a_n));$$ $$\Sec(X_{a_1,\ldots, a_n,m})=S(L,\Sec(S(a_1,\ldots, a_n)));$$
$$\Tan(X_{a_1,\ldots, a_n,m})=S(L,\Tan(S(a_1,\ldots, a_n))).$$ In particular, we get $\dim(S(X_{a_1,\ldots, a_n,m}))=2n+m+1$, $\dim(\Tan(X_{a_1,\ldots, a_n,m})=2n+m$ and
$$\deg(\Tan(X_{a_1,\ldots, a_n,m}))=2(N-2n+1)=2(N+m-\dim(\Tan(X_{a_1,\ldots, a_n,m}))+1).$$

In conclusion, all the rational normal scrolls as above  are varieties of {\it minimal tangential degree}  with $\Tan(X)\subsetneq\Sec(X)$, showing that the lower bound is sharp.
\end{exam}

\begin{cor}\label{curvesmintangdeg} Let $X\subset\PP^{N}$, $N\geq 3$, be an irreducible non-degenerate curve
such that $\deg(\Tan(X))=2(N-\dim(\Tan(X))+1)$. Then $X\subset\PP^N$ is a rational normal curve of degree $N$.
\end{cor}
\begin{proof} Since $N\geq 3$ and since $X$ is non-degenerate, then $\dim(\Tan(X))=2$ and  $\dim(\Sec(X))=3$. Moreover, since $\Tan(X)$ is ruled
by the tangent lines (they are  general fibers of the Gauss map of the surface $\Tan(X)$) we deduce $\tau(X)=1$. Then the projection of $X$
from $N-2$ general points in it is an irreducible curve $\widetilde X\subset\PP^2$ with $\tau(\widetilde X)=2$,
i.e. it is  a smooth conic. Then $X\subset\PP^N$ is an irreducible rational curve of degree $N$ and
hence a smooth rational curve of degree $N$.
\end{proof}

In the next section we shall construct examples of smooth irreducible varieties $X^n\subset\PP^N$ with $n\geq 2$
of minimal tangent degree having  $\dim(\Sec(X))=2n+1$ and $\tau(X)$ arbitrary large, which are not rational
normal scrolls.

The relation between $X$ and $\Tan(X)$ is not so strict in the sense that there might exist varieties $Y\supset X$ such
that $\Tan(Y)=\Tan(X)$. For example taking other subvarieties of $Y$, we may produce infinitely many varieties
with the same tangent variety. If this occurs, then $\Sing(\Tan(X))\supseteq Y\supsetneq X$. Let us generalize  
 the constructions that already appeared in Examples \ref{Verra}.

\begin{exam}\label{cones:omega>1} Let  $n\geq 2$, $N\geq 2n+1$ and  $Y^{n+1}\subset\PP^N$ be a non-degenerate
variety such that $\widetilde Y=\overline{\mathcal G_Y(Y)}\subset \GG(n+1,N)$ has dimension $n-1$, i.e. $Y$ is a developable scroll
in planes. For $\widetilde y\in\widetilde Y$ general, let $\PP^2_{\widetilde y}:=\overline{\mathcal G_Y^\inv(\widetilde y)}$ be a general fiber.
Let $X^n\subset Y\subset\PP^N$ be a non-degenerate variety of dimension $n$, not contained in $\Sing(Y)$, not developable 
and such that it intersects the general plane $\PP^2_{\widetilde y}$ 
along a curve $C_{\widetilde y}\subset\PP^2_{\widetilde y}$ of degree $b\geq 1$. Observe that the last condition is equivalent to  $\overline{\mathcal G_Y(X)}=\widetilde Y$.

Let $T_qX$ be a general tangent space to $X$.  If $\widetilde q=\mathcal G_Y(q)$, then  $q\in C_{\widetilde q}$ and
$$\PP^{n+1}=T_qY\supset T_qX.$$
Since $T_{y_1}Y=T_{y_2}Y$ for $y_i\in \PP^2_{\widetilde q}$,  we deduce that
\begin{equation}\label{inclTx}
\overline{\bigcup_{q\in U\subseteq C_{\widetilde q}}T_qX}\subseteq T_yY=\PP^{n+1}
\end{equation}
with $y\in \PP^2_{\widetilde q}$ general and fixed and $U\subseteq C_{\widetilde q}$ an open subset.

Since $X$ is not developable and since $\widetilde y=\mathcal G_Y(y)=\widetilde q$ is general, the tangent spaces
to $X$ along $C_{\widetilde y}$ are not constant. Hence equality holds in \eqref{inclTx} and 

$$C_{\widetilde y}^{\#}:=\overline{\bigcup_{u\in U\subseteq C_{\widetilde y}}[T_uX]}\subset(T_yY)^*$$ is a curve of degree $\tau\geq1$ such that
 through a general point of $T_yY$ there pass $\tau\geq 1$ tangent spaces to $X$ at smooth points of $X$ in $C_{\widetilde y}$.   
 We have thus proved  that $\Tan(X)=\Tan(Y)$ and that $\tau(X)\geq \tau\geq 1$. Moreover, $\tau(X)=1$ if and only if $C_{\widetilde y}\subset\PP^2_{\widetilde y}$ is a line. In particular, $$2n=\dim(\Tan(X))=\dim(\Tan(Y))$$ and  $W:=\Tan(Y)$  has
 a developable ruling by linear spaces of dimension $m\geq 3$ birationally parametrized by $\widetilde W^{2n-m}\subset \mathbb G(m,N)$. Indeed, a general  $\PP^m_{\widetilde w}$ of the ruling of $\Tan(Y)$ cuts $Y$ along 
 the plane $\PP^2_{\widetilde y}$  and $X$ along the plane curve $C_{\widetilde y}$.  Hence, for $s\in T_yY$ general, the tangent space  $T_s\Tan(Y)$
 is constant along $\langle s, \PP^2_{\widetilde y}\rangle$ because $T_s\Tan(Y)\supseteq  T_yY$ and $T_yY$ is constant along $\PP^2_{\widetilde y}$.
 If $n=2$, then necessarily $m=3$ and $\Tan(X)\subset\PP^N$ is a developable scroll in $\PP^3$'s over $\widetilde Y$.

Let us focus on the following special case. Let  $Z^{n-1}\subset {\PP}^{N}$ be a degenerate  variety 
spanning a linear space $V$ of dimension $N-2$ such that $\dim(\Tan(Z))=2n-2$ and let
 $L\subset \PP^ {N}$ be a line such that $V\cap L=\emptyset$.
Then the cone $Y^{n+1}:=S(L,Z)\subset\PP^N$ is a developable scroll in planes with $\widetilde Y$ birational to $Z$.
In other words, the projection from $L$ onto $Z$ can be identified with $\mathcal G_Y$ and $Z$ is birational to $\widetilde Y$.
In this case  $\Tan(Y)=S(L,\Tan(Z))$ also holds and $m=2n+1-\dim(\widetilde{\Tan(Z)})\geq 3$.

We claim that $\tau(Z)=1$ implies
\begin{equation}\label{tau(X)tau}
\tau(X)=\tau.
\end{equation}

Indeed, if $q\in \Tan(X)=\Tan(Y)=S(L,\Tan(Z))$ is a general point, $q':=\pi_L(q)$ and the tangent space to $X$ at $x\in X_{\reg}$ passes  through
$q$, then $\pi_L(x)$ is a smooth point of $Z$ and $\pi_L(T_xX)=T_{\pi_L(x)}Z$  passes through $q'\in \Tan(Z)$. Since $\tau(Z)=1$ and since $q'$ is general,
we have $\pi_L(x)=z$ for every such $x$. Hence  $x\in C_z$ and  $T_xX\subset T_xY$ for $x\in C_z$ general, yielding  $\tau(X)=\tau$.

 For $n=2$ the variety $Y^3$ is either  $S(L,Z)\subset\PP^N$
  or $S(p,\Tan(Z))\subset\PP^N$ 
 or $T^{(2)}Z\subset\PP^{N}$ where  $Z$ is an irreducible curve in $\PP^{N-2}$, $\PP^{N-1}$ or $\PP^N$, respectively (see Example \ref{dev:fam}). 
 As shown above we have $m=3$ and the Gauss image
 of $\Tan(X)=\Tan(Y)$ is birational to the Gauss image of $Z$.
 
 To construct a non-degenerate irreducible surface $X^2\subset Y^3\subset\PP^N$ with $N\geq 5$ and $Y$ as above and with $X$ not developable one can remark
 that  developable non-degenerate irreducible surfaces $X^2\subset\PP^M$ with $M\geq 4$ satisfy  $\dim(\Tan(X))=3$.  So if one fixes such a $Y^3\subset\PP^N$
 and takes a general hypersurface $Q\subset\PP^N$ of degree $d\geq 2$, then $X\subset Y\subset \PP^N$ will be an irreducible surface with a one--dimensional
 family of planes curves of degree $d\geq 2$ along which the tangent plane of $X$ at smooth points is not constant. Then, reasoning as above, $\Tan(X)=\Tan(Y)$
 has dimension four and $X$ is not developable. Such $X$'s are necessarily singular.
  \end{exam}

\section{Varieties $X^n\subset \PP^N$ with $N\geq 2n+1$ and $\tau(X)>1$}\label{sec:Roth}

In this section we consider the existence of  smooth varieties $X^n\subset\PP^N$, $N\geq 2n+1$, with  $\tau(X)\geq 2$
and  the classification of non-degenerate irreducible surfaces $X^2\subset\PP^N$, $N\geq 5$, with  $\tau(X)\geq 2$.

If $\tau(X)\geq 1$, then $\dim(\Tan(X))=2n$. If $X$ is smooth, then either $\Sec(X)=\Tan(X)$ or $\Tan(X)\subsetneq \Sec(X)\subsetneq\PP^N$.
In the first case we know  examples with $\tau(X)=2$ such as $QEL$-manifolds of type $\delta=1$. There are smooth examples for
sporadic values of $n$, coming from particular secant-defective homogenous varieties (see the last part of Example \ref{QELex} for the
description of some cases)  and some series of linearly normal varieties
for arbitrary $n\geq 2$, like projections of $\nu_2(\PP^n)\subset \PP^{\frac{n^2+3n}{2}}$
from the linear span of $\nu_2(\PP^s)$ with $\PP^s\subset\PP^n$ a linear subspace of dimension $s$ with $-1\leq s\leq n-2$.

The case $\Tan(X)\subsetneq \Sec(X)$ is also interesting and deserves a particular attention for several reasons which will become clear in the sequel.
To construct smooth varieties $X^n\subset \PP^N$ with $\dim(\Tan(X))=2n$, $\dim(\Sec(X))=2n+1$ and $\tau(X)\geq 2$, we look for smooth
varieties  containing a $(n-1)$-dimensional
family of plane curves of degree at least two such that the union of the planes is a $Y^{n+1}\subset\PP^N$ containing $X$ and developable along
those planes. Indeed, one proceeds in the opposite direction and analyzes  what are the restrictions imposed by the smoothness of $X$ to be contained
in a developable scroll in planes. At least for $n=2$ it is not difficult to
realize that the only possible candidates are the cones  $Y^{n+1}=S(L,Z)\subset\PP^N$ with $Z^{n-1}\subset\PP^{N-2}$ a non-degenerate variety (see Lemma \ref{nonsing} below).
Such varieties  effectively exist
as firstly remarked by Roth in \cite[\S 3.5]{Roth} for $n=2$ and then generalized by Ilic in \cite{Ilic.1998} to arbitrary $n\geq 2$. 

We start with the following key remark (see also large parts of \cite{Ilic.1998}), pointing out important connections with the topics previously discussed.

\begin{lem}\label{lem:LinX} Let $X^n\subset Y^{n+1}=S(L,Z)\subset\PP^N$ be a smooth non-degenerate variety with $Z^{n-1}\subset\PP^{N-2}=\langle Z\rangle\subset\PP^N$ a variety which is not a cone  and with $L\subset \PP^N$ a line such that $L\cap \langle Z\rangle=\emptyset$.
Then $L\subset X$ and $Z\subset\PP^{N-2}$ is a rational scroll, which is the external projection of a smooth rational normal scroll in $\PP^M$ with $M\geq N-2$.

If $X\subset\PP^N$ is also linearly normal, then $Z=S(a_1,\ldots, a_{n-1})\subset\PP^{N-2}$ is a smooth rational normal scroll, $\Tan(X)=S(L,\Tan(S(a_1,\ldots, a_{n-1})))$ and $\deg(\Tan(X))=2(N-2n+1)$.
\end{lem}
\begin{proof} Since $X_L\subseteq Z$, we necessarily have $L\cap X\neq\emptyset$. Also, since $$X\subsetneq S(L,X)\subseteq S(L,S(L,Z))=S(L,Z),$$ we deduce $\dim(S(L,X))=n+1$ so that a general $T_xX$ cuts $L$ in a point by Terracini's  Lemma. Hence $\dim(X_L)=n-1$, yielding  $X_L=Z$. Suppose
$L\cap X\subsetneq L$. Then the projection from $L$ would be solved by blowing-up $X$ in points of $X\cap L$, eventually infinitely near to one of the points
of intersection in $X\cap L$.
The last exceptional divisor $E_i=\PP^{n-1}$ over a point in $X\cap L$ would be sent into a linear $\PP^{n-1}\subseteq Z^{n-1}$, which is impossible because $X\subset\PP^N$ is non-degenerate (in this case $Y=S(L,Z)$ would be a linear subspace of dimension $n+1$).  This contradiction shows
that $L\subset X$. 

If $E\subset \Bl_LX$ is the exceptional divisor and if $\widetilde{\pi_L}:\Bl_LX\to Z$ is the resolution of the projection $\pi_L$ from $L$, then
$Z=\widetilde{\pi_L}(E)\subset \PP^{N-2}$. If $n=2$, then $Z\subset \PP^{N-2}$ is a rational curve being the image of $E=\PP^1$ by a non-constant morphism
of degree $b\geq 1$, where $b=\deg(\langle p,L\rangle\cap X)$ for $p\in Z$ general. 

We shall prove that $Z\subset\PP^{N-2}$ is a rational scroll arguing by induction on $n\geq 3$ (see also \cite[Propositions 4.12 and 4.13]{Ilic.1998}).
Then, letting $E=\PP(N_{L/X}^*(1))$, the restriction of $\widetilde{\pi_L}$ to $E$  is given by a globally generated linear system in $|\mathcal O_{\PP(N_{L/X}^*(1))}(1)|$.
 Suppose  
 \begin{equation}\label{splitE} N_{L/X}^*(1)=\bigoplus_{i=1}^{n-1}\mathcal O_{\PP^1}(b_i)
 \end{equation} 
 with $b_i\geq 0$. If $b_i=0$ for some $i$, then $Z\subset\PP^{N-2}$ would be  a  cone, contrary to our assumption. Hence $b_i>0$  for every $i=1,\ldots, n-1$ and $N_{L/X}^*(1)$ is very ample.
In particular, if $H\subset\PP^N$ is a general hyperplane through $L$, then $Y^{n-1}=X^n\cap H$ is a smooth variety (see also \cite[Proposition 4.9]{Ilic.1998}), contains $L$ and projects from $L$ onto a general hyperplane section $Y_L^{n-2}$ of $Z^{n-1}$. Now recall  that a linearly normal variety whose general curve section is a rational curve
is either a quadric hypersurface, a rational normal scroll or a cone over a Veronese surface in $\PP^5$ by a classical result of Bertini and del Pezzo (see for example \cite{EH}). Then $Z^{n-1}\subset\PP^{N-2}$, which is not a cone by hypothesis and which contains linear spaces of dimension $n-2\geq 1$ coming from $E$, is the external projection of a smooth rational normal scroll because its general hyperplane section $Y_L$ is a rational scroll of dimension $n-2\geq 1$ by induction.

Let $H\subset X$ be a hyperplane section. If $X\subset\PP^N$ is linearly normal, then $\dim(|H-L|)=N-2$. Since $Z=X_L$, then $\widetilde{\pi_L}:\Bl_LX\to Z\subset\PP^{N-2}$ is given by the complete linear system $|H-L|$ and $Z\subset\PP^{N-2}$ is linearly normal, yielding that $Z=S(a_1,\ldots, a_{n-1})\subset\PP^{N-2}$ is a smooth rational normal scroll. 
\end{proof}

The previous result says that smooth varieties contained in developable scrolls in planes $Y^{n+1}$ are
quite rare. For example, if $Y^{n+1}=S(L,Z)$ with $Z^{n-1}\subset\PP^{N-2}$  not a cone and  different from  an external projection of a  smooth rational normal scroll, then any irreducible $X^n\subset Y^{n+1}$ is necessarily singular. 

The manifolds  appearing in the last part of Lemma \ref{lem:LinX} share a notable geometric property, assuming they exist.

\begin{lem}\label{lem:tauXRoth} Let $a_i>0$,  $i=1,\ldots, n-1$, be such that  $a_1+\ldots+a_{n-1}=N-n$, let $X^n\subset S(0,0,a_1,\ldots, a_{n-1})=Y\subset\PP^N$  be a smooth 
variety  and let $\PP^2_z:=\langle z, L\rangle$ with $z\in S(a_1,\ldots, a_{n-1})$  general and $L=\Sing(Y)$. If $\PP^2_z\cap  X=L\cup C_z$ and  if $C_z$ is a smooth curve of degree 
$b\geq 1$, then $\tau(X)=b^2.$
\end{lem}
\begin{proof} The number $\tau$ of tangent spaces to $X$ at points of $C_z\subset\PP^2_z$ passing through the general point of 
$$\PP^{n+1}_z:=\bigcup_{u\in C_z}T_uX=\widetilde q(\mathcal P_{X|C_z})$$
is equal to $\tau(X)$ by \eqref{tau(X)tau}. Hence $$\tau(X)=c_1(\mathcal P_{X|C_z})=(K_X+(n+1)H)\cdot C_z$$ by \eqref{PXext}.

After blowing-up $L$, the strict transform $\widetilde C_z$ of $C_z$ on $\Bl_LX$ is a smooth irreducible fiber of the resolution of indeterminacy $\widetilde{\pi_L}:\Bl_LX\to Z$ of the rational map $\pi_L$. By the Adjunction Formula we get
$$K_{\Bl_LX}\cdot \widetilde C_z=2g(\widetilde C_z)-2=b^2-3b.$$ Hence 
$$K_X\cdot C_z=\pi^*(K_X)\cdot \widetilde C_z=K_{\Bl_LX}\cdot \widetilde C_z-(n-2)(E\cdot \widetilde C_z)=b^2-3b-(n-2)b=b^2-(n+1)b,$$
yielding $$\tau(X)=(K_X+(n+1)H)\cdot C_z=b^2.$$
\end{proof}

In view of the above results we shall  concentrate on the construction of  smooth $$X^n\subset S(0,0,a_1,\ldots, a_{n-1})\subset\PP^N$$ with $a_1+\ldots+a_{n-1}=N-n$ and $a_i>0$. This has already been done by Ilic in \cite{Ilic.1998}, who named these manifolds  as {\it Roth varieties}. 

\begin{rem}\label{nameRoth} Let notation be as above   and let $L=\Sing(S(0,0,a_1,\ldots, a_{n-1}))$. Ilic put  the condition $L\subset X$ in the definition of Roth variety.
This  is
a consequence of smoothness (and not of linear normality of $Z$), as shown in Lemma \ref{lem:LinX} (this has also been essentially  proved in \cite{Ilic.1998}). 

In \cite[\S 3.5]{Roth} Roth studied smooth surfaces $X^2\subset \PP^N$, $N\geq 5$, contained in three dimensional rational normal cones, i.e.  $X^2\subset S(0,a_1,a_2)\subset\PP^N$ with $a_i>0$ and $a_1+a_2=N-2$ or $X^2\subset S(0,0,N-2)\subset\PP^N$. 

A general hyperplane section of the above surfaces is a smooth curve of maximal genus (in the sense of the Castelnuovo bound for the genus of a curve; for this reason they are  also called {\it surfaces of maximal genus}, see \cite[\S 3.5]{Roth} and \cite{Ilic.1998} for a short summary of results on this and on the definition of {\it Castelnuovo varieties}).  The second case is special  in the sense of moduli (see also \cite[Remark 5.7]{Ilic.1998}). The existence of smooth varieties in $S(0,a_1,\ldots, a_{n-1})\subset\PP^N$ with $a_i>0$ and $a_1+\ldots+a_{n-1}=N-2$ is obvious (one can take a general hypersurface not passing through the vertex
of the cone, which  is a point). The other case considered here is less obvious, although already known to Roth. 
\end{rem}

Ilic proved the following characterization of Roth varieties. We formulate it  in a slightly modified form due to the previous remarks.

\begin{thm}[{\cite[Theorem~3.8]{Ilic.1998}}]\label{Roth:Ilic}   
Let $n\geq 2$ and let 
$S=S(0,0,a_1,\ldots,a_{n-1})\subset\PP^N$
with $a_1+\ldots+a_{n-1}=N-n$ and $a_i>0$.
    
For every $b\geq 1$ there exists a  Roth variety $X^n\subset S\subset\PP^N$ of degree $b(N-n)+1$, whose class in the Weil group
of divisors of $S$ is $bH+F$,  with $b\geq 1$, $H$ the class of a hyperplane section of $S$ and $F$ the class of a $\PP^n\subset S$.

 Conversely, if   $X^n \subset S\subset\PP^N$ is a smooth manifold  of degree $d$, then $b=\frac{d-1}{N-n}$ is an integer and 
 the class of $X$ as a Weil divisor on $S$ is  $bH+F$.  In particular, Roth varieties are linearly normal.
 
\end{thm}

\begin{rem}\label{Roth:char}
For an arbitrary smooth projective variety $X^n\subset\PP^N$ of degree $d$ the linear system
$|(d-n-2)H-K_X|$ is base-point free as pointed out by Mumford (see the introduction of \cite{Ilic.1998}
for a discussion of this property).

By \eqref{splitE} and by the arguments in the proof of Lemma \ref{lem:LinX}, for a Roth variety
$X^n\subset S(L,Z)\subset\PP^N$ of degree $d=b(N-n)+1$ we have 
 $$\deg(N^*_{L/X})+(n-1)=\sum_{i=1}^{n-1}b_i=b\cdot \deg(Z)=b\cdot (N-n),$$
 yielding 
 $$K_X\cdot L=\deg(N^*_{L/X})-2=b(N-n)-n-1$$
and 
$$((d-n-2)H-K_X)\cdot L=0.$$
The non-ampleness of  $|(d-n-2)H-K_X|$  characterizes (isomorphic projections of) Roth
varieties by the main result of \cite{Ilic.1998}.
\end{rem}

\begin{rem}\label{Roth:sec}

For a Roth variety of degree $d=b(N-n)+1$ a general plane curve $C_z=\langle L,z\rangle \cap  X$ is smooth and  has degree $b\geq 1$. 

 Let $X^n\subset\PP^N$, $N\geq 2n+2$,  be a Roth variety of dimension $n\geq 2$. Then through a general point $p\in\Sec(X)$ there pass
$b^2$ secant lines. Indeed, if $p\in\langle p_1,p_2\rangle$, $p_i\in X$, and if $C_{z_i}\subset X$ are the plane curves of degree $b$ passing through $p_i$, then 
$$\langle C_{z_1}, C_{z_2}\rangle=\PP^3_p.$$
Hence through $p$ there pass $b^2$ secant lines to $C_{z_1}\cup C_{z_2}$ by B\' ezout's  Theorem. Conversely, all  secant lines to $X$ passing through $p$ project from $L$ onto a secant line
to $S(a_1,\ldots, a_{n-1})\subset \PP^{N-2}$ passing through $r:=\pi_L(p)\in \Sec(S(a_1,\ldots, a_{n-1}))$. Since through $r$ there passes a unique secant line $\langle z_1,z_2\rangle$ to $S(a_1,\ldots, a_{n-1})$,
then every secant line to $X$ passing through $p$ is contained in $\PP^3_p$ and hence it is one of the secant lines described above. The linear space $T_p\Sec(X)$ is
tangent to $X$ along $C_{z_1}\cup C_{z_2}$ although $\dim(\Sec(X))=5$ (see \cite{Chiantini.Ciliberto.wd} for a modern treatment of this phenomenon, called {\it weakly defectiveness} by Chiantini and Ciliberto).

When $p_2$ tends to the general point $p_1$, the secant line $\langle p_1,p_2\rangle$ becomes a tangent line to $X$ at $p_1$ and  the curve $C_{z_1}\cup C_{z_2}$ degenerates into a double structure on $C_{z_1}$. The line $\langle z_1,z_2\rangle$ becomes a tangent line $M$ to $S(a_1,\ldots, a_{n-1})$ at $z_1$ and $\PP^3_p$ degenerates into
$\langle L,M\rangle=\PP^3$. Hence the $b^2$ tangent spaces passing through a general point  $t\in\Tan(X)$ are determined by $b^2$ normal directions to the double structure
on $C_{z_1}$ passing through $t$.

When $Y^3=S(p, TZ)$ or when $Y=T^{(2)}Z$  one can construct (singular) irreducible surfaces $X\subset\PP^N$, $N\geq 6$, with  $\mu(X)=1$  and $\tau(X)\geq 2$
and also smooth irreducible surfaces with $\tau(X)=1$ and $\mu(X)\geq 1$ (see also \cite{Chiantini.Ciliberto.wd}).
\end{rem}

We now proceed to the classification of non-degenerate irreducible surfaces $X\subset \PP^N$ with $\tau(X)\geq 2$ and with $N\geq 5$. We start with a general remark.

\begin{pr}\label{Contact locus of TX in X is 1-dimensional}
    Let $N\geq 2n+1$ and let $X^n\subset\PP^N$  be a  non-degenerate variety  such that $\tau(X)\geq 2$. Let 
       $$\PP^{\varrho(X)}_w=\overline{\mathcal G_{\Tan(X)}^\inv
(\mathcal G_{\Tan(X)}(w))}$$ with $w\in \Tan(X)$ general. Then $\varrho(X)\geq 2$ and  $\PP^{\varrho(X)}_w\cap X$ contains an equidimensional curve $C_w$ and $T_w\Tan(X)$ is tangent to $X$ along $C_w\cap X_{\reg}$. 
\end{pr}
\begin{proof}
     Let $w\in \Tan(X)$ be general. By hypothesis, there exist $x_1,x_2,\dots,x_{\tau(X)}\in X_{\reg}$ such that $w\in T_{x_i}X$ and $T_{x_i} X\neq T_{x_j} X$ for all $i\neq j$. Also, we can assume $\langle w,x_i\rangle\cap \langle w,x_j\rangle=\{w\}$ since a general projective tangent space to $X$ is not bitangent by the non developability of $X$. We have $T_w \Tan(X) =T_z \Tan(X)$ for $z\in \langle w, x_i \rangle \cap \Tan(X)_{\reg}$ by  
     Proposition \ref{TX is developable}.  Hence, $\langle w,x_i ,x_j\rangle\subseteq\PP^{\varrho(X)}_w$, yielding $\varrho(X)\geq 2$.

 Let $ U_w\subset \Tan(X)_{\reg}$ be an open analytic ball centered at $w$ such that $$q_{|q^\inv(U_w)}:q^\inv(U_w)\rightarrow U_w$$ is a holomorphic $\tau(X)$-covering map. Also, let $q^\inv(U_w)=\bigsqcup_{i=1}^{\tau} U_w^i$, where $U_w^i\subseteq T_X$ is an open analytic neighborhood of $(x_i,w)$ isomorphic to $U_w$ via $q$. We can shrink $U_w$ as much as necessary so that  $x_i\notin U_w$ for every $i$. Lastly, define $C_w^i:=q_{|U_w^i}^{-1}(\langle w,x_i \rangle\cap U_w)\subset U_w^i$.
    
    Fix $i\in\{1,2,\dots,\tau(X)\}$. Let $z\in \langle w,x_i \rangle\cap U_w$ be general and let $x\in X_{\reg}\setminus x_i$ be such that $(x,z)\in C_w^j$ for some $j\neq i$, which in particular implies $z\in T_x X$. 
 Then $\overline{p(C_w^j)}\subset X$ (Zariski closure) is a curve passing through $x_j$ and $x$ and we let
 $$C_w=\bigcup_{j=1}^{\tau(X)}\overline{p(C_w^j)}\subset X.$$
  We claim that this curve is contained in $\PP^{\varrho(X)}_w$.
 Indeed if $i\in\{1,\ldots,\tau(X)\}$ is fixed, if $z\in \langle w,x_i \rangle$ is general, then $\langle z, x \rangle \cap \Tan(X)_{\reg}\neq \emptyset$ by the generality of $w$ and $\langle z, x\rangle\subset T_xX$ by definition of $C_w^i$. So,  for $y\in \langle z, x \rangle \cap \Tan(X)_{\reg}$ general, we have 
  $$T_y\Tan(X) =T_z\Tan(X)=T_w \Tan(X),$$ yielding  $y\in \PP^{\varrho}_w$.  Hence $\langle z,x\rangle \subset \PP^{\varrho(X)}_w$ and $C_w\subset \PP^{\varrho(X)}_w$ by the generality of $x$, proving the  claim.    Let notation be as above. If $x\in C_w\cap X_{\reg}$, then $T_w\Tan(X)=T_z\Tan(X)\supset T_xX$ by the last part of Proposition \ref{TX is developable}, concluding the proof.
 \end{proof}

Let us remark that that  by monodromy of the covering we have that either $C_w$ is irreducible or there exist
 $\tau(X)\geq 2$ distinct irreducible components of $C_w$. Moreover,  $\deg(C_w)\geq 2$ holds also
 in the irreducible case because $X\subset\PP^N$ is not a linear subspace.
\medskip

Suppose now $N\geq 5$ and that $X^2\subset\PP^N$ is an irreducible surface with $\tau(X)\geq 2$. Since $\dim(\Tan(X))=4$ and since $\varrho(X)\geq 2$ we only have  two possibilities:
either $\varrho(X)=2$ or $\varrho(X)=3$. We treat the two cases separately.

\begin{lem}\label{varrho=2} Let $N\geq 5$ and let $X^2\subset\PP^N$ be a non-degenerate irreducible surface such that $\tau(X)\geq 2$ and $\varrho(X)=2$. Then $N=5$,
$X\subset\PP^5$ is projectively equivalent to the Veronese surface $\nu_2(\PP^2)\subset\PP^5$ and $\tau(X)=2$. 
\end{lem}
\begin{proof} Let $w\in\Tan(X)$ be general and let $C_w\subset \PP^2_w\cap X$ be the curve constructed in Proposition \ref{Contact locus of TX in X is 1-dimensional}. The curves $C_w\subset X$  vary  in a two dimensional family $\mathcal C$ birationally parametrized by the Gauss image of $\Tan(X)$.
These curves cover $X$ and through two general points $x_1,x_2\in X$ there passes at least one plane curve $C\in\mathcal C$, yielding $\deg(C)\geq 2$.
The general secant line $L=\langle x_1, x_2\rangle$ to $X$ is not trisecant so that $\deg(L\cap  C)=2$, showing that  the curves $C_w$ are conics.
Through a general point $p\in L$ there pass infinitely many secant lines to $X$, yielding  $\dim(\Sec(X))\leq 4$. Therefore $\Tan(X)=\Sec(X)\subsetneq\PP^N$, a general point $p\in L$ is also a general point of $\Tan(X)$ and the entry locus $\Sigma_p\subset X$ is a conic (see for example \cite[Theorem 1.1]{Chiantini.Ciliberto.wd}), which coincides with $C_p$.  Since $\tau(X)\geq 2$, the conic $C_p$ is irreducible. 
A classical result by C. Segre assures that  a  non-degenerate irreducible surface $X^2\subset\PP^N$ with $N\geq 5$  and containing a 
two--dimensional family of irreducible conics
is projectively equivalent to a Veronese surface in $\PP^5$ (see for example the proof of \cite[Theorem 3.4.1]{Russo.Book2016}).
\end{proof}
\begin{rem}
More generally, the above proof can be easily adapted to prove that if $N\geq 2n+1$ and if $X^n\subset\PP^N$ is a non-degenerate variety with $\tau(X)\geq 2$, then
$\Tan(X)=\Sec(X)$, $X$ is a $QEL$-variety of type $\delta=1$ and $\tau(X)=2$. 
\end{rem}
We now consider the case $\varrho(X)=3$.

\begin{lem}\label{varrho=3} Let $N\geq 5$ and let $X^2\subset\PP^N$ be a non-degenerate irreducible surface such that $\tau(X)\geq 2$ and $\varrho(X)=3$. Then
$X^2\subset Y^3\subset \PP^N$ where $Y^3$ is the locus of the focal planes of the developable family of $\PP^3$'s determined
by $\Tan(X)$. 

In particular, $X^2\subset\PP^N$ is contained in a three-dimensional developable scroll in planes such that a general plane of the scroll
cuts $X$ along a plane curve of degree at least two.
\end{lem}
\begin{proof} Let $W=\overline{\mathcal G_{\Tan(X)}(\Tan(X))}\subset \GG(4,N)$ be the Gauss image of $\Tan(X)$, which is a developable scroll in $\PP^3$'s over a curve birational to $W$.
Let $\PP^3_t\subset \Tan(X)$ be the closure of the fiber of the Gauss map of $\Tan(X)$ over a general $t\in W$ and let 
 $C_t\subset \PP^3_t\cap X$  be the curve constructed in Proposition \ref{Contact locus of TX in X is 1-dimensional}.
 
 The locus of foci in  $\PP^3_t$ of the family of developable $\PP^3_t$'s is  $\PP^2_t:=\PP^3_t\cap (\PP^3_t)^\prime$.  Since $\Tan(X)$ is not a cone with vertex a fixed plane,
the $\PP^2_t$'s move and describe a variety  $$Y^3=\overline{\bigcup_{t\in U\subseteq W}\PP^2_t}$$ of dimension three, which is an irreducible component of $\Sing(\Tan(X))$. Also the ruling of $\PP^2_t$'s is developable and such that $T_qY=\PP^3_t$ for $q\in \PP^2_t$ general. In particular $\Tan(Y)=\Tan(X)$.

From the parametrization of $T_z\Tan(X)$ for $z\in T_xX$ general in \eqref{diffG},  we deduce that $T^{(2)}_xX=T_z\Tan(X)$  for $z\in\PP^3_t$ general and $x\in C_t$ general. In particular $T_x^{(2)}X$ is constant along $C_t$ and equal to $\PP^4_t:=T_y^{(2)}Y$ with $y\in \PP^2_t$ general. Let $\mathcal G_X^{(2)}:X\map \mathbb G(4,N)$ be the second Gauss map of $X$, associating to a general $x\in X$ the point in $\mathbb G(4,N)$ corresponding to $T_x^{(2)}X$. Then $\overline{G_X^{(2)}(X)}=W$ and a dense open subset of $C_t$ is contained in  the fiber of $G_X^{(2)}$ over $t$, which is smooth by generic smoothness.  We claim that $C_t\subset \PP^2_t$.

Without loss of generality we can suppose $N=5$. By abusing notation we shall also indicate by $t$  a local parameter at a general point of the curve $W$. By the generality of $t\in W$ and of $x\in C_t$, we have that $t$ is a local equation of $C_t$ at $x$ because it coincides with the fiber of $G^{(2)}_X$ through $x.$  The hyperplane $\PP^4_t=T^{(2)}_xX$ cuts $X$ 
near $x$ along the divisor $3\cdot C_t$, which is locally defined at $x$ by the pull-back of the local equation $t^3$. Then the infinitesimally near $\PP^3_t=\PP^4_t\cap (\PP^4_t)'$ cuts $X$  near $x$ along the divisor defined by $\frac{\partial}{\partial t}t^3$, i.e. along an open subset of the divisor $2\cdot C_t$ containing $x$. Finally, $\PP^2_t=\PP^3_t\cap (\PP^3_t)'=\PP^4_t\cap (\PP^4_t)'\cap (\PP^4_t)^{''}$ cuts $X$  near $x$ along the divisor defined by $\frac{\partial^2}{\partial t^2}t^3$, i.e. along an open subset of $C_t$ containing $x$. The generality of $x$ assures that a general point
in each irreducible component of $C_t$ has a neighbourhood contained in $\PP^2_t$, proving the claim and thus concluding the proof.
\end{proof}

\begin{lem}\label{nonsing} Let $N\geq 5$ and let $X^2\subset Y^3\subset\PP^N$  be a smooth linearly normal non-degenerate irreducible surface  with $Y^3$ a developable scroll in planes. Then
$Y=S(0,0,N-2)$ and $X^2\subset \PP^N$ is a Roth surface of degree $d=b(N-2)+1$ with $b\geq 1$. 
\end{lem}
\begin{proof} Let $Y^3=\Tan^{(2)}C\subset\PP^N$ with $C\subset\PP^N$ a non-degenerate irreducible curve. Let $\overline C$ be the normalization of $C$, let $\mathcal L$ be the pull-back of $\mathcal O_C(1)$ to $\overline C$, let $\mathcal P(\mathcal L)$ be the  sheaf of  principal parts of $\mathcal L$ on $\overline C$ and let $\mathcal P^2(\mathcal L)$ be the  sheaf of second order parts of $\mathcal L$ on $\overline C$. The latter two are locally free sheaves of ranks two and three. The bundle $\mathcal P(\mathcal L)$ can also be defined via the extension:
$$0\to \Omega_{\overline C}^1\otimes \mathcal L\to \mathcal P(\mathcal L)\to \mathcal L\to 0$$
corresponding to  $c_1(\mathcal L)\in H^1(\Omega^1_{\overline C})$. There exists  also
a surjective homomorphism  
$$ \mathcal P^2(\mathcal L)\to \mathcal P(L),$$
yielding an embedding $\PP(\mathcal P(L))\subset \PP(\mathcal P^2(\mathcal L))$. This is just the abstract realization of $\Tan(C)\subset \Tan^{(2)}C$.

If $\PP^N=\PP(V)$, the surjection $V\otimes \mathcal O_{\overline C}\to \mathcal L$ induced by pull-back of sections
yields surjections of $V\otimes \mathcal O_{\overline C}$ onto  $\mathcal P(\mathcal L)$ and  $\mathcal P^2(\mathcal L)$. The latter one
defines  a morphism $g:\PP(\mathcal P^2(\mathcal L))\to \PP^N$, whose image is $\Tan^{(2)}C$. The morphism $g$ is a bijection between the fibers
over a general  point of $\overline C$ and the corresponding osculating plane to $C\subset\PP^N$. The morphism $g$ ramifies along $\PP(\mathcal P(L))$, whose image via $g$ is $\Tan(C)$,
and $\Tan^{(2)}C$ has cuspidal singularities along $\Tan(C)$. Indeed, letting $\pi:\PP(\mathcal P^2(L))\to \overline C$ be the structural morphism, the reinterpretation  of the analysis in \S \ref{dev:fam} is that, taking a  point $r\in \PP^1_t\subset \PP^2_t=\pi^\inv(t)$, $t\in C$ general, the image of  $(dg)_r:t_r\PP(\mathcal P^2(\mathcal L))\to t_{g(r)}\PP^N$ is $t_{g(r)}\PP^2_t$ and the kernel of $(dg)_r$ maps isomorphically onto $t_{\pi(r)}\overline C$. 

Any smooth $\overline X\subset\PP(\mathcal P^2(\mathcal L))$, different from a $\PP^2_t$ from the ruling, dominates $\overline C$,
 intersects $\PP(\mathcal P(\mathcal L))$ in a curve $D$ and intersects $\PP^2_t$ in a smooth  plane curve $A_t$  with $t\in C$  general.
Let $r\in A_t\cap \PP^1_t$, $t\in C$ general (i.e. $r\in D$ general).  Then $t_r\overline X\cap t_r\PP^2_t=t_rA_t\subset t_r\PP^2_t$ so that  $t_r \overline X$ has one dimensional
image via $(dg)_r$ and $g(\overline X)$ will be singular at $g(r)$.
The case $Y=S(p,\Tan(C))$ is similar and left to the reader.

Suppose   $X\subset Y=S(L,C)\subset\PP^N$ with $C\subset\PP^{N-2}\subset\PP^N$ a  non-degenerate irreducible curve in its linear span and with
$L\subset\PP^N$ a line disjoint from $\langle C\rangle=\PP^{N-2}$. From Lemma \ref{lem:LinX} we deduce that $C\subset\PP^N$ is a rational normal
curve, that  $Y=S(0,0,N-2)$ and that $X$ is a Roth surface. These surfaces exist by Theorem \ref{Roth:Ilic}.
\end{proof}

We can now collect some relevant consequences of the previous results.

\begin{thm}\label{class:surf:taugeq2} Let $N\geq 5$ and let $X^2\subset \PP^N$ be a non-degenerate irreducible surface with $\tau(X)\geq 2$.
Then one of the following holds:

\begin{enumerate}
\item  $\varrho(X)=2$, $N=5$ and $X$ is projectively equivalent to a Veronese surface;
\medskip
\item

$\varrho(X)=3$ and $X^2\subset Y^3\subset\PP^N$
with $Y^3\subset\PP^N$ a developable scroll in planes which intersect the surface in curves of degree at least 2.
\end{enumerate}
\end{thm}

\begin{cor}\label{smooth:taugeq2} Let $N\geq 5$ and let $X^n\subset \PP^N$ be a smooth  linearly normal  irreducible surface with $\tau(X)\geq 2$.
Then one of the following holds:
\begin{enumerate}  
\item $\tau(X)=2$, $N=5$ and $X$ is projectively equivalent to the Veronese surface;
\medskip
\item $\tau(X)=b^2$ and 
$X\subset\PP^N$ is a Roth surface of degree $d=b(N-2)+1$ for some $b\geq 2$.
\end{enumerate}

In particular, a smooth linearly normal irreducible surface $X^2\subset\PP^N$ with $N\geq 5$, $\Tan(X)\subsetneq \Sec(X)$ and $\tau(X)\geq 2$
is a Roth surface of degree $d=b(N-2)+1$ with $b\geq 2$ and $\tau(X)=b^2$.
\end{cor}
\begin{proof} It is enough to apply Lemma \ref{nonsing} to the conclusions of Theorem \ref{class:surf:taugeq2}.
\end{proof}

The formula \eqref{degTan} below has been derived many times under much stronger assumptions, see for example \cite[Theorem 5.3]{Cattaneo}.

\begin{cor}\label{cor:PX} Let $N\geq 5$ and let $X^2\subset\PP^N$ be a smooth irreducible non-degenerate surface and let  $$\widetilde q:\PP(\mathcal P_X)\to \Tan(X)\subset\PP^N$$ be the tautological morphism.
Then the following conditions are equivalent:
\vskip 0.2cm
\begin{enumerate}
\item either $X\subset\PP^5$ is a Veronese surface or $X\subset\PP^N$ is (the projection of) a Roth surface.
\vskip 0.2cm
\item $\widetilde q$ is not birational.
\end{enumerate}

In particular, if  $X^2\subset\PP^N$ is a smooth  linearly normal surface, not a Roth or a Veronese surface, and if $H\subset X$
is a hyperplane section, then 
\begin{equation}\label{degTan}
\deg(\Tan(X))=6\deg(X)+4H\cdot K_X +2K_X^2-12\chi(\mathcal{O}_X).
\end{equation}
\end{cor}
\begin{proof} The first part is clear by Corollary \ref{smooth:taugeq2}. Hence the hypothesis in the second part assure $\tau(X)=1$ so we have
$$\deg(\Tan(X))=\omega_2(X)=\deg(\mathcal O_{\PP(\mathcal P_X)}(1)^4)$$
by \eqref{SegreclasseFormula}. Then formula  \eqref{degTan} follows
by a standard computation of Chern/Segre classes, taking into account the exact sequence \eqref{PXext}.
\end{proof}

\bibliographystyle{plain}
\bibliography{bib}{}

\begin{thebibliography}{123}

\bibitem{Catanese}
F.~Catanese, 
\newblock {\it On Severi’s proof of the double point formula}, 
\newblock Comm. Algebra {\bf 7} (1979), 763--773.

\bibitem{Cattaneo}
A.~Cattaneo, 
\newblock {\it On the degree of the tangent and of the secant variety to a projective surface}, 
\newblock Adv. Geom. {\bf 20} (2020), 233--248.

\bibitem{Chiantini.Ciliberto.wd}
L.~Chiantini, C.~Ciliberto,
\newblock {\it Weakly defective varieties}
\newblock Trans. Amer. Math. Soc. {\bf 354} (2001), 151--178.

\bibitem{Chiantini.Ciliberto.korder}
L.~Chiantini, C.~Ciliberto,
\newblock {\it On the concept of $k$-secant order of a variety}
\newblock J. London Math. Soc. {\bf 73} (2006), 436--454.


\bibitem{Chiantini.Ciliberto.Russo}
L.~Chiantini, C.~Ciliberto, F.~Russo
\newblock {\it On secant defective varieties, in particular of dimension 4}
\newblock Math. Z. {\bf 311} (2025), article 51,  https://doi.org/10.1007/s00209-025-03836-1


\bibitem{Ciliberto.Russo.2006}
C.~Ciliberto, F.~Russo,
\newblock {\it Varieties with minimal secant degree and linear systems of
              maximal dimension on surfaces}
\newblock Adv. Math. {\bf 200} (2006), 1--50.



\bibitem{Ciliberto.Sernesi.2010}
C.~Ciliberto, E.~Sernesi,
\newblock {\it Projective geometry related to the singularities of theta divisors of
  {J}acobians},
\newblock Boll. Unione Mat. Ital. {\bf 9} (2010), 93--109.

\bibitem{EH} D.~Eisenbud, J.~Harris,
\newblock {\it On varieties of minimal degree}, 
\newblock Alg. Geometry Bowdin 1985, 
Proc. Symp. in Pure Math. {\bf 46} (1987), 3--13.

\bibitem{Fischer.Piontkowski.Book2001}
G.~Fischer, J.~Piontkowski,
\newblock {\it Ruled varieties. An introduction to algebraic differential geometry},
\newblock Advanced Lectures in Mathematics. Friedr. Vieweg \& Sohn,
  Braunschweig, 2001.

\bibitem{Griffiths.Harris.1979}
P.~Griffiths, J.~Harris,
\newblock {\it Algebraic geometry and local differential geometry},
\newblock Ann. Sci. \'Ecole Norm. Sup. {\bf 12} (1979), 355--452.

\bibitem{Ilic.1998}
B,~Ilic,
\newblock {\it Geometric properties of the double-point divisor},
\newblock Trans. Amer. Math. Soc. {\bf 350} (1998), 1643--1661.


\bibitem{IR}
P.~Ionescu, F.~Russo,
\newblock {\it Conci-connected manifolds},
\newblock J. reine angew. Math. (Crelle's Journal)  {\bf 644}  (2010), 145--157.


\bibitem{Italiani}
M.~Italiani,
\newblock {\it Sulle congruenze di piani di $S_4$},
\newblock Boll. Unione Mat. Ital.  {\bf 13} (1958), 105--111.


\bibitem{Johnson}
K.W.~Johnson,
\newblock {\it Immersion and embedding of algebraic varieties},
\newblock Acta Math.  {\bf 140} (1978), 49--74.

\bibitem{Pedreira}
M.~Pedreira, L.~Sol\' a--Conde,
\newblock {\it Classification of congruences of planes in $\PP^4_\mathbb C$. I.},
\newblock Geom. Dedicata  {\bf 85}  (2001), 69--83.

\bibitem{PR1}
L.~Pirio, F.~Russo,
\newblock {\it Varieties $n$ covered by curves of degree $\delta$},
\newblock Comm. Math. Helv.   {\bf 88}  (2013), 715--757.

\bibitem{PR2}
L.~Pirio, F.~Russo,
\newblock {\it The $XJC$-correspondence},
\newblock J. reine angew. Math. (Crelle's Journal)  {\bf 716}  (2016), 229--250.

\bibitem{PR3}
L.~Pirio, F.~Russo,
\newblock {\it Quadro-quadric Cremona maps and varieties 3-connected by cubics: the semi--simple part and the radical},
\newblock Int. J. Math  {\bf 24} no. 13 (2013), 1350105 (33 pages).


\bibitem{Roth}
L.~Roth,
\newblock {\it On the projective classification of surfaces},
\newblock Proc. London Math. Soc.  {\bf 42} (1937), 142--170.



\bibitem{Russo.Book2016}
F.~Russo.
\newblock {\it On the geometry of some special projective varieties},
  Lecture Notes of the Unione Matematica Italiana, vol.~18,
\newblock Springer, 2016.

\bibitem{Segre1}
C.~Segre,
\newblock {\it Preliminari di una teoria delle varietà luogo di spazi},
\newblock Rend. Circolo Mat. Palermo   {\bf 30} (1910), 87--121.

\bibitem{Segre2}
C.~Segre,
\newblock {\it Sui fuochi di $2^o$ ordine dei sistemi infiniti di piani e sulle curve iperspaziali con una doppia infinità di piani plurisecanti},
\newblock Rend. R. Acc. Naz. Lincei    {\bf 30} (1921), 67--71.



\bibitem{Severi}
F.~Severi,
\newblock {\it Sulle intersezioni delle varietà algebriche e sopra i loro caratteri
e singolarità proiettive},
\newblock  Mem. Accad. Scienze di Torino, S. II, {\bf 52} (1902), 61--118.

\bibitem{Terracini} 
A.~Terracini,  
\newblock{\it Alcune questioni sugli
spazi tangenti e osculatori ad una variet\' a}, I, II, III,
\newblock Selecta Alessandro Terracini, pages 24--90, Edizioni Cremonese, Roma,
1968.

\bibitem{Vermeire}
P.~Vermeire, {\it Some results on secant varieties leading to a geometric flip
  construction}, Comp. Math. {\bf 125} (2001), 263--282.

\end{thebibliography}
\end{document}